\documentclass[12pt]{article}
\usepackage{amsmath}
\usepackage{amsfonts}

\usepackage{graphicx}
\usepackage{color}

\def\epsfig#1{}

\newtheorem{theorem}{Theorem}
\newtheorem{corollary}[theorem]{Corollary}
\newtheorem{definition}[theorem]{Definition}
\newtheorem{example}[theorem]{Example}
\newtheorem{lemma}[theorem]{Lemma}
\newtheorem{proposition}[theorem]{Proposition}
\newtheorem{remark}[theorem]{Remark}
\newenvironment{proof}[1][Proof]{\textbf{#1.} }{\ \rule{0.5em}{0.5em}}

\newcommand{\dom}{\mathop{\rm Dom}}

\newcommand{\spt}{\mathop{\rm spt}}

\newcommand{\p}{\partial}

\newcommand{\R}{\mathbf{R}}

\newcommand{\ub}{\underline}
\newcommand{\ts}{\textstyle}
\newcommand{\ds}{\displaystyle}
\newcommand{\cm}{c}
\newcommand{\cmd}{\cm_\delta}
\renewcommand{\rho}{r}

\newcommand{\blt}{{b'_{\theta'}}}
\newcommand{\bet}{{b_{\theta}}}

\newcommand\tus{{\tilde u_\sigma}}
\newcommand\tvs{{\tilde v_\sigma}}

\begin{document}

\author{Alice Erlinger\thanks{Laboratoire Dieudonn\'e, Universit\'e de Nice, 
Parc Valrose 06108 Nice Cedex 2 France 
{\tt alice.erlinger@gmail.com}},$^\ddagger$ 
Robert J. McCann\thanks{Department of Mathematics,
University of Toronto, Toronto Ontario M5S 2E4 Canada,
{\tt mccann@math.toronto.edu}}, 
Xianwen Shi$^\S$, \\ 
Aloysius Siow$^\S$ and Ronald Wolthoff\thanks{
Department of Economics, 
University of Toronto, 
{\tt 
xianwen.shi@utoronto.ca, siow@chass.utoronto.ca, ronald.wolthoff@utoronto.ca}}}


\title{
Academic wages and pyramid schemes: a mathematical model%
\thanks{
RJM 
thanks the University of Nice Sophia-Antipolis,
the Becker-Friedman Institute for Economic Research, and
the Stevanovich Center for Financial Mathematics at the University of Chicago
for their kind hospitality
during various stages of this work.  He acknowledges partial support of his research by
Natural Sciences and Engineering Research Council of Canada Grant 217006-08,
and, during Fall 2013 while he was in residence at the Mathematical Sciences Research 
in Berkeley, California, by the National Science
Foundation under Grant No. 0932078 000.
We are grateful to Gary Becker, Yann Brenier and Rosemonde Lareau-Dussault
for fruitful conversations.
\copyright 2014 by the authors.}}
\date{\today}

\maketitle

\begin{abstract}
This
 paper analyzes a steady state matching model interrelating
the education and labor sectors.  In this model,  a heterogeneous population
of students match with teachers to enhance their cognitive skills.  As
adults, they then choose to become workers, managers, or teachers,
who match in the labor or educational market to earn wages by producing output.
We study the competitive equilibrium which results from the steady
state requirement that the educational process replicate the same
endogenous distribution of cognitive skills among adults in each generation
(assuming the same distribution of student skills).
We show such an equilibrium can be found by solving an
infinite-dimensional linear program and its dual.  We analyze the structure
of our solutions, and give sufficient conditions for them to be unique.
Whether or not the educational matching is positive assortative turns out to 
depend on convexity of the equilibrium wages as a function of ability, suitably parameterized; 
we identity conditions which imply this convexity.
Moreover,
due to the recursive nature of the education market,  it is a priori conceivable
that a pyramid scheme leads to greater and greater discrepancies in the
wages of the most talented teachers at the top of the market. Assuming each
teacher teaches $N$ students,  and contributes a fraction $\theta \in]0,1[$ to their
cognitive skill, we show
a phase transition occurs at $N\theta=1$,  which determines whether or not the wage
gradients of these teachers remain bounded as market size grows,
and make a quantitative prediction
for their asymptotic behaviour in both regimes: $N\theta \ge 1$ and $N\theta<1$.
\end{abstract}


\bigskip




















\section{Introduction}

It is an economic truism that prices are 
determined primarily 
by what the market will bear. For example, executive compensations  in large firms
may appear excessive when measured against average employee wages, but are 
often justified by arguing that they are 
determined competitively by the market. 
To understand what levels
of compensation a large market will or won't bear, it is therefore tempting
to ask questions such as: Can the ratio of the CEO's wages over the average
wage in a firm be expected to tend to infinity or a finite limit, as the
size of the firm grows without bound? The answer to such a question may be
expected to depend on various aspects of the organization of the firm, such
as the number of levels of management separating the CEO from the average
worker, and the number of managers at each level. This organizational
structure may itself be determined by market pressures --- within the
constraints of feasible technology.

In this paper we investigate an analogous question set in the context of the
education market, rather than that of a firm. That is, we investigate how
the wages of the most sought after gurus relate to those of the average
teacher. The education market is special in various ways. It is stratified
into many different levels or streams which interact with each other, with a
range of qualities available in every stream. Moreover, what it produces is
human capital, the value of which is determined by the broader market for
skills of which the education market is itself a small part. Thus there is a
feedback mechanism in the education market, owing to the fact that those
individuals who choose to become teachers participate at least twice in the
market: first as consumers and later as producers, putting to work the
skills previously acquired in this market to generate human capital for the
next generation. It is this feedback mechanism which is responsible for many
of the results we describe; it leads to the formation of an educational
analog for a pyramid scheme, in which teachers at each level of the pyramid
attempt to extract as much as they can from their students future earnings,
in the form of tuition. The question this time is whether or not the large
market limit leads to wages which display singularities at the apex of the
pyramid.

We address this question using a variant of a steady state matching model
introduced by four of us to analyze the coupling of the education and labor
markets \cite{McCannShiSiowWolthoff13p}. We proposed this model not only to
provide a microeconomic foundation which allows to compare and contrast
different sectors, but to examine interdependencies and the different roles
played by communication and cognitive skills in each of them. An unexpected
conclusion was that --- as in much simpler (single stage, single sector)
models \cite{ShapleyShubik72} \cite{GretskyOstroyZame92} 
\cite{ChiapporiMcCannNesheim10}, 
competitive equilibrium matching patterns for a heterogeneous steady state
population can be found as the optimal solution to a planner's problem
taking the form of a linear program; see also \cite{Ekeland10}. 
The questions raised in the present
manuscript will be addressed through a rigorous analysis of the resulting
linear program and its solutions, including criteria for existence,
uniqueness, singularities, and a detailed description of the matching
patterns which can arise. A remarkable feature is that this simple model
leads to the emergence of a hierarchical structure in the education sector,
with fewer and fewer individuals at the top of the market earning higher and
higher wages. A detailed exploration of this structure proves necessary to
resolve the question of under what conditions these wages turn out to
display singularities. 
An analogous hierarchy was explored 
by Becker and Murphy in the context of a 
steady growth model \cite[\S VII]{BeckerMurphy92} 
quite different from ours. 

The education market is also unusual in many ways that our
model does not capture. For example, non-pecuniary considerations are
important for both teachers and students, and schools are often not operated
on a for-profit basis; however, in our model we assume all participants
maximize their expected monetary payoff. In addition, education markets
(tuitions, for example) are heavily regulated, but here we abstract away
all regulation restrictions. The goal of this paper, therefore, is not
to provide a realistic account of how teachers' compensations are determined
in the market, but rather to elucidate a feedback mechanism that is
potentially important in determining wage compensation in education and 
other markets, and to provide a tool to solve matching models that
incorporate this feedback mechanism with the potential to encompass
multi-dimensional individual attributes. 

In the present model, we assume that the communication skills are
homogeneous over the entire population, hence deal with a population having
a single dimension of heterogeneity plus parameters, rather than the
multiple dimensions of heterogeneity in \cite{McCannShiSiowWolthoff13p}.
Hence, the model here can be viewed as a limiting case of multidimensional
models in which the range of heterogeneities becomes narrow in all
dimensions but one. There are two benefits from this simplifying assumption.
First, it greatly simplifies our analysis. Second, the resulting model is a
minimal departure from the classical matching model of one dimension of
heterogeneity. We will show that this small departure actually generates
results very different from the standard one-dimensional models of e.g. 
 Lucas \cite{Lucas78} or Garicano \cite{Garicano00}.

As in \cite{McCannShiSiowWolthoff13p}, we model
communication skills as the number of students a teacher can teach or the
number of workers a manager can manage, which is often referred to as
\textquotedblleft span of control\textquotedblright . In particular, we
assume that each teacher can teach $N>1$ students. We
use $\theta \in ]0,1[$ to represent the extent to which a
teacher's cognitive skills get transmitted to each of their students.
Similarly, $N^{\prime }>0$ and $\theta ^{\prime }\in ]0,1[$ represent the
number of workers each manager can manage in the labor market, and the
extent to which a manager's cognitive skill enhances the productivity of his
or her workers. All market participants have the same $N$
in the education market and the same $N^{\prime }$ in the labor
market, but they differ in the cognitive skills $k$ which are
assumed to be continuously distributed over the interval 
$\bar{K}:=[\underline{k},\bar{k}]\subset \R$. As a result, the linear program is
infinite-dimensional, and the analysis is complicated by a lack of a priori
bounds which could be used to show that equilibrium wages or payoffs exist
for the model. Moreover, a pyramid can form in the education sector,
enhancing the wages of the most skilled teachers. It is not obvious whether
or not this pyramid structure can lead to unbounded wage behavior.
Our analysis suggests it does not,  but leads to unbounded wage gradients instead.

We begin by elucidating a convexity property which allows us to derive the
existence of equilibrium wages as solutions to an (infinite-dimensional)
linear program. This convexity is reminiscent of 
that discussed by Rosen in his investigation of superstars~\cite{Rosen81}. 
More surprisingly, after addressing uniqueness and properties of these wages
and the matches they induce, we go on to show that the model 
exhibits a phase transition, depending on the product 
of each teacher's capacity $N$ for students times their teaching
effectiveness $\theta $: 
the wage gradients diverge at the highest skill type if and only if $N\theta
\geq 1$. 
When $N\theta>1$, the divergence is proportional to $|\bar{k}-k|^{-\frac{\log \theta }{\log N}%
-1}$ as $k\rightarrow \bar{k}$. Only by integrating this divergence can we
conditionally show wages tend to a finite limit at $\bar{k}$ which --- in
the large market limit --- becomes independent of the size of the population
being modelled. 

Although wage singularies for teachers may appear counter-factual,
or at least modest compared to wage singularities for managers in the real world,
this discrepancy between prediction and observation is easily explained by the fact that
our model allows for only one layer of managers but a potentially unbounded number of 
layers of teachers.  Thus a top teacher improves the cognitive skills of each of
their $N$ students who go on to be top teachers or managers. 
A good manager improves the productivity of each of their $N^{\prime }$ 
supervised workers. Thus, already in a two-layer hierarchy, a top
teacher indirectly makes a large number $N\times N^{\prime }$ of workers
more productive. Since the number of layers of the educational hierarchy is 
endogenous to the model and can be very large, the impact of gurus on the productivity
of their direct and indirect students and workers can accumulate very substantially.

The term \emph{phase transition} is borrowed from statistical physics, where
it refers to a sharp threshold in parameters (such as temperature)
separating qualitatively different behavior (such as liquid from solid). In
that context, the non-smoothness 
arises from a continuum limit which admits 
approximation by finite dimensional models depending smoothly on the same
parameter(s). By analogy, if our continuum of agent types could be
approximated using finitely many agent types, we would expect to restore
smooth dependence on the parameters $N$ and $\theta$, but this smoothness
(i.e.\ the wage gradients) would not admit control uniform in the number of
types. In statistical physics, it is often the case that the critical
exponents of the singularities (such as $\frac{\log \theta}{ \log N}$ above)
do not vary over a wide class of models, a phenomenon known as universality. 
In the present context, we observe that the exponent $\frac{\log \theta}{
\log N}$ governing growth of the wage gradients is universal in the sense
that it does not 
depend on various details of the model, such as the exact form of the
production functions, or the input distribution of student skills, at least
within the classes of such data considered hereafter.

The remainder of this manuscript is organized as follows. In the first 
section and subsections we lay out the model, and its variational reformulation in
terms of a planner's problem and its dual. We have argued in \cite%
{McCannShiSiowWolthoff13p} that solutions to these infinite-dimensional
linear programs represent competive equilibria; see also the announcement 
\cite{McCann14p}. In a second section and subsections we address the existence,
uniqueness and properties of these solutions. Even the existence of
equilibrium wages in this model is rather non-trivial, and goes beyond the
range of validity of any statement of the second welfare theorem that we
know. Standard arguments concerning existence of an optimal matching and
absence of a duality gap are relegated to an appendix, which is logically
independent of the rest of the analysis. Lemma \ref{L:adult tail bound} is
also logically independent of the remaining analysis, and its first
assertion is actually required at some earlier points in the text.

\subsection{The model: competitive equilibria}

Let us begin by describing our unidimensional variant of the model
first introduced by \cite{McCannShiSiowWolthoff13p}. Consider an economy
populated by risk-neutral individuals who each lives for two periods.
Individuals, when they are young, enter the education market as students. In
the subsequent period as adults, they enter the labor market to become
teachers in schools, or workers or managers in firms. Both the education
market and the labor market are competitive. There is free entry for both
schools and firms. Hence, the tuition fees a school collects from students
are just enough to cover the wage of its teacher, and a firm's output
exactly covers the wages of its employees (workers and mangers). All
individuals do not discount. The lifetime net payoffs of individuals are
equal to the sum of their labor market plus non-labor market earnings minus tuition costs.
Individuals choose what occupation to pursue and who to match with in each of the two 
markets  to maximize their net payoffs.

Each individual is endowed with two kinds of skills, a
communication skill ($N>1$ or $N^{\prime }>1$) which is
fixed throughout their lifetime, and an initial cognitive skill $a$ 
which can be augmented through education. As in~\cite{McCannShiSiowWolthoff13p}, we
assume that individuals differ in their initial cognitive skills $a$. 
In contrast to \cite{McCannShiSiowWolthoff13p}, we assume that individuals
share the same communication skills. By attending schools in the first
period, individuals can augment their initial cognitive skills $a$ 
to their adult cognitive skill $k$. Let $A=[\underline{a},\bar{a}[$ 
with $-\infty <\underline{a}<\bar{a}<+\infty $ denote
the range of students' initial cognitive skills $a$,
and $K=[\underline{k},\bar{k}[$ or rather its closure $\bar{K}$ 
the range of adult human capital $k$. Ability or human
capital refers to cognitive skill in both cases, and we occasionally
use the variable names $a$ and $k$ interchangeably for convenience. 
For the model discussed
here, taking $K=A$ will not cost any generality, nor will the
normalization $\underline{a}=\underline{k}=0$. 

The production functions in the education market and in the labor
market are described as follows. We assume the cognitive skill $z(a,k)$
acquired by a student of ability $a\in A$ who studies with a teacher of
ability $k\in K$ is given by the weighted average $z(a,k)=(1-\theta
)a+\theta k$ of their abilities, with weight $\theta \in ]0,1[$. We also
assume the productivity 
$b_{L}((1-\theta ^{\prime })a+\theta ^{\prime }k)$ of a worker with adult
cognitive skill $a$ supervised by a manager of skill $k$ is given by a
convex increasing function $b_{L}\in C^{1}\left(\bar{K}\right)$ of another such
average, this time with weight $\theta ^{\prime }\in ]0,1[$. Notice
that abilities $a$ and $k$ here are measured on a
logarithmic scale relative to the conventions of \cite%
{McCannShiSiowWolthoff13p}, a reparameterization which is crucial for
exposing the sense in which the equilibrium wages may turn out to be convex.

We allow for the possibility that cognitive skill $z$ attained through
education has value $cb_{E}(z)$ in addition to the wage earning potential it
confers, where $c\geq 0$ is a dimensionless parameter and $b_{E}\in C^{1}(%
\overline{A})$ is another convex increasing function. The choice $%
b_{E}(k)=e^{k}=b_{L}(k)$ with $\theta =\frac{1}{2}=\theta ^{\prime }$
corresponds to the motivating example from \cite{McCannShiSiowWolthoff13p};
more generally we assume $b_{E}$ and $b_{L}$ and their first two derivatives
have positive lower bounds 
\begin{eqnarray}
\label{utility bound 0}
0<\underline{b}_{E/L} &=&b_{E/L}(0)  \\
\label{utility bound 1}
0<\underline{b}_{E/L}^{\prime } &=&b_{E/L}^{\prime }(0) \\
\label{utility bound 2}
0<\underline{b}_{E/L}^{\prime \prime } &=&\inf_{k}b_{E/L}^{\prime \prime
}(k),
\end{eqnarray}%
where $\underline{b}_{E/L}^{\prime \prime }$ is defined as the largest
constant for which $b_{E/L}(k)-\underline{b}_{E/L}^{\prime \prime }|k|^{2}/2$
is convex on $\bar{K}$. We hope strict positivity of the analogous
quantities will be inherited by the equilibrium payoffs $u$ and $v$. 

Notice that what is being produced in each sector is different: 
in the labor and non-labor sectors we have not specified the service or
goods which are being produced, except that they take adult cognitive skills
as their input (communication skills entering through possible dependence of 
$c$ on parameters such as $N$ and $\theta$); in the education sector it is
adult cognitive skills which are being produced, taking student and teacher
cognitive skills as their inputs.  The dimensionless constant $c \ge 0$
measures the non-labor utility, 
if any, of individual attainment of cognitive skills 
relative to labor productivity; it replaces the marital utility used in
early drafts of \cite{McCannShiSiowWolthoff13p}.

Let a probability measure $\alpha \geq 0$ on $\bar{A}$ represent the
exogenous distribution of student abilities, and let $\mathop{\rm spt}\alpha 
$ denote the smallest closed subset of $\bar{A}$ carrying the full mass of $%
\alpha $. Taking $A$ smaller if necessary ensures $\mathop{\rm spt}\alpha $
contains both $\underline{a}$ and $\bar{a}$. 
Our problem is to find a pair Borel measures $\epsilon \geq 0$ on $\bar{A}%
\times \bar{K}$ and $\lambda \geq 0$ on $\bar{K}\times \bar{K}$, such that $%
\epsilon $ represents the \emph{educational} pairing of students with
teachers, and $\lambda $ represents the \emph{labor} pairing of workers with
managers, along with a pair of payoffs or wage functions $u,v:\bar{K}%
\longrightarrow \lbrack 0,\infty ]$ representing the net lifetime expected
utility $u(a)$ of a student with ability $a$, and the wage $v(k)$ paid to an
adult of ability $k$, which together constitute a competitive equilibrium $%
(\epsilon ,\lambda ,u,v)$. Roughly speaking, this means the matchings $%
\epsilon ,\lambda $ must clear the market at each generation in a
steady-state, and the payoffs $u$ and $v$ must be large enough to be stable,
yet small enough that in combination with $(\epsilon ,\lambda )$ they
satisfy a budget constraint.

Since we are interested in a steady state model,  we assume the distribution of student abilities $\alpha$ on $\bar A$
is the same at each generation,  and coincides with the left marginal
\begin{equation}\label{student marginal}
\epsilon^1 = \alpha
\end{equation}
of the educational pairing $\epsilon \ge 0$ of student and teacher abilities.
Here $\epsilon^{1} = \pi^{1}_\# \epsilon$ and $\epsilon^{2} = \pi^{2}_\#\epsilon$
denote the left and right projections of $\epsilon$ through $\pi^{1}(a,k)=a$
and $\pi^{2}(a,k)=k$, representing the respective distributions of student and teacher abilities.
  Similarly $\lambda^{1}$ and $\lambda^{2}$ will denote
the left and right marginals of the labor pairing $\lambda$, representing the distribution of
worker and manager skills.  The steady state constraint requires that the educational pairing $\epsilon$ of students with adults reproduce the current distribution of adult skills at the next generation:
\begin{equation}\label{steady state}
\lambda^{1} + \frac{1}{N'} \lambda^{2} + \frac1{N} \epsilon^{2} =
z_\#\epsilon,
\end{equation}
where the expression on the left represents the sum of the current
distributions of worker, manager and teacher skills;
the latter have been scaled by $N'$ and $N$ respectively, to reflect the fact that each manager manages $N'$ workers,
and each teacher teaches $N$ students, so comparatively fewer managers and teachers are required.
The symbol $\kappa:=z_\# \epsilon$ on the right represents the distribution of future adult skills 
resulting from the educational pairing $\epsilon$; it is given
by the push-forward of $\epsilon$ through the 
map $z:\bar A \times \bar K \longrightarrow \bar K$ representing the educational technology, and assigns mass
$\kappa[B] := \epsilon[z^{-1}(B)]$ to each set $B \subset \bar K$.

The marginal constraint \eqref{student marginal} forces $\epsilon$ and hence
$\kappa = z_\#\epsilon$ to be probability measures, like $\alpha$.
  The workers form a fraction $(1-\frac1N)/({1+\frac1{N'}})$ of the population, 
coinciding with the total mass of $\lambda$.
The restriction $K=A$ costs no generality,
since we are in a steady state,  and since our education technology
satisfies $z(a,a)=a$, whence $z(\ub a,\ub k) = \ub k$ and
$z(\bar a,\bar k) = \bar k$.

Letting $v(k)$ denote the wage commanded by an 
adult of skill $k$,  and $u(a)$ the net lifetime utility of a student of ability $a$,  both 
must satisfy the stability conditions
\begin{eqnarray}
\label{stability education}
 u(a) + \frac1N v(k) \ \ge&      \cm b_E(z(a,k)) + v(z(a,k)) & {\rm and}\\
v(a) + \frac 1{N'}v(k) \  \ge&  b_L((1-\theta')a+\theta'k)  & {\rm\ on}\ \bar A \times \bar K.
\label{stability labor} 
\end{eqnarray}
The constraint \eqref{stability labor}
enforces stability of matchings
in the labor sector.  If the reverse inequality held,  $N'$ adults with
skills $a$ and one with skill $k$ would 
abandon
their occupations to form $N'$ worker-manager pairs each producing enough output $b_L$  
to improve all $N'+1$ adults' wages.
Similarly 
\eqref{stability education} is a stable matching
condition for the education sector.  The lifetime net utility 
of a student with
cognitive skill $a$ plus the tuition $v(k)/N$ paid by each student of a teacher with skill $k$ must exceed $a$'s lifetime earnings plus any other benefits derived from cognitive skills which would have resulted had he (and $N-1$ of his clones) chosen to study with 
$k$.  We can also regard the stability constraints \eqref{stability education}--\eqref{stability labor} as combining to ensure
each adult of type $k$ in the population chooses the profession (worker, manager, or teacher) and partners (manager, workers, or students, respectively) which maximize their wage $v(k)$ on the labor market.

Finally,  the budget constraint asserts that 
equality holds $\epsilon$-a.e.\ in \eqref{stability education},  and
$\lambda$-a.e.\ in \eqref{stability labor}.
In other words,  the productivity 
$b_L((1-\theta')a + \theta'k)$ of $\lambda$-a.e.\ manager-worker pair $(a,k)$
which actually forms is sufficient to pay the worker's wage 
plus a fraction $1/N'$ of the manager's
salary. Similarly,  $\epsilon$-a.e.\ student-teacher pairing $(a,k)$ which forms must produce an adult whose
earnings $v(z(a,k))$, supplemented by any additional utility $c b_E(z(a,k))$ derived from the skill
$z(a,k)$ he acquires, must add up to the net lifetime utility 
which remains to the student after paying tuition
equal to his share $v(k)/N$ of his teacher's earnings.

To complete the specification of the model, we need to say in what class of functions the payoffs
$u,v$ must lie. 
Since we wish to allow for the possibility that the payoffs $u,v:K \longrightarrow [0,\infty]$
become unbounded at the upper end $\bar k$ of the skill range,  it is convenient to define
$A=K=[0,\bar k[$ as a half open interval.  We shall consider payoffs from the feasible set $F_0$ 
consisting of pairs $(u,v) = (u_0+u_1,v_0+v_1)$ 
satisfying \eqref{stability education}--\eqref{stability labor}
which differ from
bounded continuous functions
$u_0,v_0 \in C(\bar A)$ by non-decreasing functions 
$u_1,v_1:\bar A \longrightarrow [0,\infty]$.
If $v$ takes extended real values, we also require
\begin{equation}\label{stability bound}
\ts \frac{N}{N-1} (u(k) - \cm b_E(k)) \ge v(k) \ge \frac{N'}{N'+1} b_L(k) >0 \qquad {\rm\ on}\ \bar K,
\end{equation}
which otherwise follows from $a=k$ in \eqref{stability education}--\eqref{stability labor}.
We often require $u$ and $v$ to be {\em proper}, meaning lower semicontinuous and not identically infinite.
This costs little generality,  since
when \eqref{stability education}--\eqref{stability bound} hold for non-negative functions $(u,v)$,
they continue to if $u$ and $v$ are replaced by their lower semicontinuous hulls.

A {\em competitive equilibrium} refers to a pair of measures $\epsilon,\lambda \ge 0$ and functions
$(u,v) \in F_0$ satisfying \eqref{student marginal}--\eqref{stability bound}
plus the budget constraint 
\begin{equation}\label{budget constraint}
\mbox{\rm equality holds $\epsilon$-a.e.\ in \eqref{stability education},  and
$\lambda$-a.e.\ in \eqref{stability labor}}
\end{equation}
relating $(\epsilon,\lambda)$ to $(u,v)$.
The economic idea behind this definition is that no
individual agent (nor any group of agents which is small relative to the size of the market)
can improve their outcome by choosing to match otherwise than as prescribed
by $\epsilon$ and $\lambda$. 
Here $\epsilon$ represents
an assignment of $N$ students to each teacher, and reproduces the current distribution
of adult skills in the next generation,
starting from the given distribution $\alpha$ of student skills and
educational technology $z(a,k) = (1-\theta)a + \theta k$;
the future earnings plus any non-labor utility received by the $N$ students exactly add up to
their net lifetime utilities, plus the salary of the teacher.
Similarly, $\lambda$ represents an assignment of $N'$ workers to each manager, the productivity of these worker-manager teams exactly sufficing to
pay the respective wages of each team member.  Both the educational and the labor markets clear, and the stability constraints guarantee no
adult would prefer an occupation other than the one he or she has been assigned,
nor to work with anyone other than the partners prescribed by $(a,k) \in \spt\lambda$
in the case of workers or managers,  or by $\epsilon$ in the case of teachers.
Similarly,  each pair $(a,k) \in \spt \epsilon$ represents a student of ability $a$,
who cannot improve his net lifetime payoff by training with any teacher other than
the one of skill $k$ that he is paired with under $\epsilon$.

\subsection{The planner's problem and its dual}
Shapley and Shubik's basic insight is that stable matching problems 
with transferable utility
have a variational reformulation using linear programs and their duals.
In \cite{McCannShiSiowWolthoff13p} we observe that this insight extends
from the familiar single-stage, single-sector setting of 
\cite{ShapleyShubik72} \cite{GretskyOstroyZame92} and \cite{ChiapporiMcCannNesheim10},
to steady-state multi-sector models such as the one introduced above. 
Denoting our education and labor market technologies by
$\bet(a,k) = b_E((1-\theta)a + \theta k)$ and $\blt(a,k)= b_L((1-\theta')a + \theta' k)$,
the quartuple $(\epsilon,\lambda,u,v)$ forms a competitive equilibrium
if and only if $(u,v)$ attain the infimum
\begin{equation}
LP_* := \inf_{(u,v) \in F_0}
\int_{[0,\bar a]}
u(a) \alpha(da)
\label{three line dual}
\end{equation}
over \eqref{stability education}--\eqref{stability bound}, 
while $(\epsilon,\lambda)$ attain the supremum
\begin{equation}\label{three line primal}
LP^* := \max_{
{\epsilon \ge 0\ {\rm and}\ \lambda \ge 0\ {\rm on}\ [0,\bar a]^2
\atop {\rm satisfying}\ \eqref{student marginal}-\eqref{steady state}}}
\int_{[0,\bar a]\times [0, \bar k]}
\!\!\!\!\!\! 
[c \blt(a,k) \epsilon(da,dk) + \bet(a,k) \lambda(da,dk)].
\end{equation} 
We shall henceforth refer to $(u,v) \in F_0$ as {\em optimal} if it attains
the infimum \eqref{three line dual},  and to $(\epsilon,\lambda)$
as {\em optimal} if it attains the supremum \eqref{three line primal}.
Whereas the notion of competitive equilibrium relates
$(u,v)$ to $(\epsilon,\lambda)$ through \eqref{budget constraint},  
one can discuss {\em optimality} of $(u,v)$ without referring to
$(\epsilon,\lambda)$, and vice-versa.  This is the first of many advantages 
conferred by our Shapley-Shubik-like reformation of the problem at hand.
 
We often use $\alpha(u)$ as a shorthand notation to denote the integral appearing
in \eqref{three line dual}, which represents the average student's net lifetime utility.
Similarly, $c \epsilon(\bet) + \lambda(\blt)$ denotes the
argument appearing in the supremum \eqref{three line primal},  and represents
the total (non-labor + labor) utility produced by the pairings $\epsilon$ and $\lambda$.
Thus if equilibrium wages $(u,v) \in F_0$ exist, they minimize the expected lifetime utility
of students subject to the stability constraints.
Similarly, any equilibrium matches maximize
the utility $c \epsilon(\bet) + \lambda(\blt)$ being produced our model's two
sectors in each generation, subject to the market-clearing constraints 
\eqref{student marginal}--\eqref{steady state} in steady-state.   
The latter can be interpreted as a
social planner's problem;  it is also the linear program dual to \eqref{three line dual}.
Satisfaction of the budget constraint
\eqref{budget constraint} follows from the absence of a duality gap: 
the fact $LP_*=LP^*$, which is established below under the technical hypothesis that $\alpha$
satisfy a doubling condition at the top skill type $\bar a$, meaning there exists
$C<\infty$ such that
\begin{equation}\label{doubling}
\int_{[\bar a-2\Delta a, \bar a]} \alpha(da) \le C\int_{[\bar a - \Delta a,\bar a]} \alpha(da)
\end{equation}
for all $\Delta a>0$.  A surprisingly delicate part of the proof is the inequality
$LP^* \le LP_*$ shown in Proposition \ref{P:easy duality};  the rest of the duality argument
reproduced in Appendix \ref{S:no duality gap} is quite standard.

The variational characterization given by \eqref{three line dual}--\eqref{three line primal}
is our starting point for the further analysis for the payoffs $(u,v)$ and 
matchings $(\epsilon,\lambda)$ we seek. 
To show such competitive equilibria exist,  it is enough to establish the infimum and supremum are attained.  
Attainment of the planner's supremum is standard,  as recalled in 
Appendix \ref{S:no duality gap}.  
It is less straightforward to show that the infimum \eqref{three line dual}
is attained, and to elucidate the properties of the extremizers for either problem.
A continuity and compactness argument is complicated by the fact that the wage function $v$
appears on {\em both} sides of the education sector stability constraint,  and has no obvious upper bound except  
 in $L^1(\bar A,\alpha)$; c.f.\ \eqref{stability bound}.


When minimizers $(u,v)$ exist,  it is useful to know as much structural information
as we can about them,  in order to analyze the properties of the corresponding equilibrium matches.  In the cases for which we have been able to deduce the existence
of minimizers,  they turn out to be non-negative, non-decreasing, {\em convex} functions 
of $a \in [0,\bar a]$.
The fact that the monotonicity and convexity of $u$ and $v$
survive limits is crucial to the analysis.  
Indeed, our existence strategy is
to first show \eqref{three line dual} is minimized under the additional
assumption of convexity and monotonicity for $u$ and $v$,  and then to show
that this additional constraint does not bind for the minimizing $(u,v)$,
which must therefore optimize the original problem of interest.
In the absence of an atom at the top skill type, $\alpha[\{\bar k\}]=0$,
it seems possible a priori that
both $u(k)$ and $v(k)$ diverge to $+\infty$ as $k\to\bar k$,  without violating
boundedness of the expected value $LP_* = \alpha(u)$.  Although
Theorem \ref{T:phase transition} 
tends to rule out this possibility,
giving conditions instead for the gradients $u'(a)$ and $v'(k)$ to diverge,  for the intermediate analysis it is useful to let $A=[0,\bar k[=K$
denote a half-open interval where we can assume $u$ and $v$ are real valued.

In addition to $(N,\theta)$ and $(N',\theta')$, dimensionless parameters
such as  $\bar b'_L/ \ub b'_L \ge 1$ and 
$\cm \ge 0$ govern the behavior displayed by the model. 
Here $\ub b_L'$ is from \eqref{utility bound 1} and
\begin{eqnarray*}
\bar b'_{E/L} = b'_{E/L}(\bar k) = \sup_{k \in K} b'_{E/L}(k) 
\end{eqnarray*} 
so $\bar b'_L/\ub b'_L$ indexes the relative impact of an increase in skill on labor productivity at the top versus the bottom of the skills market,  while $c$ measures the
relative importance of any other satisfactions derived from cognitive skills apart from the returns to labor which they help to enhance.  Such satisfactions could be intrinsic,
or they could represent externalities that cognitive skills and education provide,
such as social status or --- as in early drafts of \cite{McCannShiSiowWolthoff13p} --- marital prospects.
We can also remove this effect from the model by setting $\cm=0$.  
However,  to implement the existence strategy outlined above, 
it turns out to be technically easier to analyze the case $\cm>0$ first,  
and then take the limit $\cm \to 0$ if desired.  Many but not all of our structural results 
such as uniqueness, specialization, and positive assortativity also survive this limit;  see Proposition~\ref{P:specialization} and Theorem \ref{T:unique} for example.


We shall also investigate occupational specialization by cognitive skill,
showing $ \min\{N'\theta', N \theta\} \ge  \bar b_L'/\ub b_L'$
implies that the highest types become teachers,  while the lowest types become either workers or teachers,
but not managers.
More refined statements appear in Proposition \ref{P:specialization}.
For continuously distributed skill types, we show that a pyramid can form in the education sector,  sometimes leading to divergence of wage gradients at the highest
skill type when $N\theta \ge 1$,  meaning the span of control at each node in the pyramid 
is large enough.
More explicitly,  under suitable conditions
Theorem~\ref{T:phase transition} asserts that as $k\to \bar k$,
$$
v'(k) \sim \left\{
\begin{array}{ccc}
const|\bar k - k|^{-1-\frac{\log \theta}{\log N}} &{\rm for}& N\theta > 1, \\
{c \bar b_E'}/{(\frac1{N\theta} -1)} 
&{\rm for}& N\theta <1,
\end{array}
\right.
$$
so a phase transition occurs at $N\theta=1$.
A less involved investigation of an analogous pyramid structure was given by
Becker and Murphy \cite[\S VII]{BeckerMurphy92},
in a different production model incorporating the cost of acquiring knowledge 
and assuming steady-growth as opposed to steady-state.
To produce a similar pyramid in the labor sector,  our model would need to be modified to
permit managers to manage other managers --- as in 
Garicano~\cite{Garicano00} 
with Rossi-Hansberg~
\cite{GaricanoRossiHansberg06} ---
instead of being forced to manage only workers whose productivity is inherently limited.  
If such a modification to our model could be achieved, 
it would have the potential to complement existing models for 
executive compensation such as Gabaix and Landier's \cite{GabaixLandier08},
which rely instead on 
comparing given tail behaviors of the distributions of
company size and managerial talent.

Finally, Corollary \ref{C:characterize} characterizes the optimizers in the
primal and dual problems \eqref{three line dual}--\eqref{three line primal}.
Theorem \ref{T:unique} provides sufficient conditions for
uniqueness of $(\epsilon,\lambda)$,
and discusses in what sense $(u,v)$ are also unique.
It gives conditions
guaranteeing the optimal pairings $\lambda$ of workers with managers
and  $\epsilon$ of teachers with students
are positive assortative in cognitive
skills,  meaning $\spt \lambda$ and $\spt \epsilon$  are non-decreasing 
subsets of the plane.  This monotonicity is intimately tied to the
convexity of wages $v$ as a function of $k \in [0,\bar k]$ asserted above.



\section{Analysis}

\subsection{Terminology and notation}



In this section, we introduce terminology and notation that
will be useful for dealing with functions which need neither be smooth
nor bounded, and with the measures which arise naturally as their duals.

Given any convex set $B \subset \R^n$, a function
$u:B \longrightarrow \R \cup \{+\infty\}$
is said to be continuous if it is upper and lower semicontinuous.
It is said to be
{\em Lipschitz} with Lipschitz constant $L$ if
either $u$ is identically infinity or else if
$$L := \sup_{B \ni x \ne y \in B}  \frac{|u(x) - u(y)|}{|x-y|}$$
is finite.  It is said to be {\em semiconvex} with semiconvexity constant $C$
if the function $x \in B \longmapsto u(x) + C|x|^2/2$ is convex.
It is said to be {\em locally} Lipschitz (respectively semiconvex)
on $B$, if $u$ is Lipschitz (respectively semiconvex) on every compact
convex subset of $B$.  Locally Lipschitz (respectively semiconvex) functions
are once (respectively twice) differentiable Lebesgue a.e.  In addition,
locally semiconvex functions fail to be once differentiable on a set of
Hausdorff dimension at most $n-1$.  

By {\em support} of a Borel measure $\alpha \ge 0$ on $\R^m$,  we mean the smallest
closed subset $\spt \alpha\subset \R^m$ of full mass: $\alpha[\R^m \setminus \spt \alpha]=0$.
The {\em push-forward} $f_\#\alpha$ of $\alpha$ through a Borel map
$f:\R^m \longrightarrow \R^n$ is a Borel measure defined by
$(f_\#\alpha)[Z] = \alpha[f^{-1}(Z)]$ for each $Z \subset \R^n$.
We say $\alpha$ has no {\em atoms} if $\alpha[\{x\}]=0$ for each $x \in \R^m$.
A measure $\epsilon$ on $\R^2$ is said to be {\em positive assortative} if
$\spt \epsilon$ forms a non-decreasing subset in the plane: i.e.\ if
$(a'-a)(k'-k) \ge 0$ for all $(a,k),(a',k') \in \spt \epsilon$.  We use
$\alpha|_B$ to denote the restriction $\alpha|_B(Z)=\alpha[Z\cap B]$ of $\alpha$ to
$B \subset \R^m$,  and $H^n$ to denote Lebesgue measure on $\R^n$.

\subsection{The educational pyramid}

In this section, we discuss the extent to which we can expect optimizers
$(u,v)$ to the minimization \eqref{three line dual} 
to be smooth,  at least
away from the top skill type~$\bar k$.
We then apply these results to elucidate the nature of the
pyramid structure which can form in the education sector.


Given $(u,v) \in F_0$ feasible for the infimum \eqref{three line dual},
use $\blt(k',k) := b_L((1-\theta')k' + \theta' k)$ and $z(a,k) = (1-\theta)a + \theta k$
to define the wages implicitly available to an individual
of cognitive skill $k$ employed as a worker, manager, or teacher, respectively:
\begin{eqnarray}\label{wage_w}
v_w(k) &:=& \phantom{N'} \sup_{k'\in \bar K} \blt(k,k')-\ts \frac{1}{N'} v(k'),
\\ \label{wage_m}
v_m(k) &:=& N' \sup_{k' \in \bar K} \blt(k',k) - v(k'), \qquad {\rm and}
\\ \label{wage_t}
v_t(k) &:=& N \sup_{a \in \bar A}  \ts cb_E(z(a,k)) + v(z(a,k)) - u(a),
\end{eqnarray}
where we complete definition \eqref{wage_t}, and later \eqref{student wage},
with the convention
\begin{equation}
\label{infinity convention} \infty - \infty := \infty.
\end{equation}
The suprema \eqref{wage_w}--\eqref{wage_m} are attained when $u$ and $v$ are proper
(hence lower semicontinuous), and the same holds true for
 \eqref{wage_t} if, in addition, $v$ is convex non-decreasing 
(hence continuous).

Clearly feasibility \eqref{stability education}--\eqref{stability bound} implies
$v \ge \bar v := \max\{v_w,v_m,v_t\}$.  When equality holds ---
as we shall see that it does (Theorem \ref{T:minimizing wages})
for some $v$ 
minimizing \eqref{three line dual}
 --- this implies strong conclusions.
For example, $v_w$ and $v_m$ inherit Lipschitz and convexity
properties from $b_L$  by an envelope argument
(Lemma \ref{L:inherit Lipschitz and semiconvex}),
which $v$ also inherits wherever it coincides with $v_w$ or $v_m$.
Something similar is true but more subtle to verify for $v_t$
(and hence for $\bar v$) ---
because of the recursive structure built into the educational pyramid;
in \eqref{wage_t}, as opposed to \eqref{wage_w}--\eqref{wage_m},
this is manifested in the fact that the $k$ dependence in the argument of the supremum
involves the unknown function $v$.
As another example, when $N'\theta'$ and $cN\theta$ are large enough,
Proposition \ref{P:specialization} derives complete specialization of types into
low (workers), medium (managers),  and high (teachers).
This at least tells us the role of
$\kappa$-a.e.\ agent,  leaving the distribution
$\kappa = \kappa_w + \kappa_m  + \kappa_t$ of adults as the
only unknown.  Here $\kappa_w = \lambda^{1}$, $\kappa_m = \lambda^{2}/N'$ and
$\kappa_t = \epsilon^{2}/N$
are measures representing the distribution of worker, manager, and teacher types,
and have respective masses
$\kappa_w[\bar K] = \frac{(N-1)N'}{N(N'+1)}$, $\kappa_m[\bar K]=\frac{N-1}{N(N'+1)}$
and $\kappa_t[\bar K]=\frac1N$.
If $c=0$ but $\min\{N'\theta',N\theta\} \ge \bar b_L'/\ub b_L'$,
the same proposition yields more subtle conclusions.

A first insight into the educational pyramid is provided by the following example.

\begin{example}[Gurus] Fix the number of students each teacher can teach or the
number of workers each manager can manage to be $N=N'=10$.  If our probability
measure $\kappa$ represents the skill distribution for a population of $110$ adults,
90 of them will be workers,  managed by 9 managers,  and 11 of them will be teachers.
Nine of these $11=9+1+1$ will specialize in teaching workers,  one in teaching teachers,
and one in teaching a combination of 9 managers and 1 teacher.
We may remember this with the mnemonic $110= 90+9+(9+1+1)$.
On the other hand, if $\kappa$ represents the skill distribution for a population of
$11000 = 9000+900+(900+90+(90+9+(9+1+1))$ adults,  9000 of them will be workers,
managed by 900 managers, while 1100 of them will be teachers.  Of these,
900 will teach workers, 90 will teach managers, and 110 will teach teachers.
Within these 110,  there is further specialization as before: 90 will teach teachers
who teach workers,  9 will teach teachers who teach managers,  and 11 will teach
teachers who teach teachers.  Within this 11,  9 teach worker-teacher-teachers,
1 teaches manager-teacher-teachers, and 1 teaches only teacher-teacher-teachers.
These last two may be thought of as {\em `gurus'}. One of the questions at stake
is whether the salaries of these gurus can grow without bounds as the population size grows.
\end{example}

Next we recall without proof a well-known result which can be proved as
in~\cite{GangboMcCann96}:

\begin{lemma}[Upper envelopes inherit derivative bounds]
\label{L:inherit Lipschitz and semiconvex}
If $f:A \times K \longrightarrow \R$ is locally Lipschitz in $a \in A$,
uniformly in $k \in K$,  then $g(a) = \sup_{k \in K} f(a,k)$ is locally Lipschitz
and for each $\delta>0$ we have the bounds
$$
\inf_{k \in K, |a'-a|<\delta} f_a(a',k) \le
g'(a) \le \sup_{k \in K, |a'-a|<\delta} f_a(a',k)
$$
in 
the pointwise a.e.\  
senses.
Similarly,  if $f$ is locally semiconvex in $a \in A$,
uniformly in $k \in K$,  then $g(a)$ is locally semiconvex and obeys the
bound
$$
g''(a) \ge \inf_{k \in K, |a'-a|<\delta} f_{aa}(a',k)
$$
in the same senses.
Here $f_a := \frac{\p f}{\p a}$ and $f_{aa} := \frac{\p^2 f}{\p a^2}$.

If $f(a',\cdot)$ extends upper semicontinuously to $\bar k$ for some $a' \in A$,
allowing  $f(a',\bar k) = -\infty$ as a possible value,
there exists $k' \in \bar K$ such that $g(a') = f(a',k')$;
if $g(a)$ is differentiable at
$a' \in ]\ub a,\bar a[$, the envelope theorem then yields
$g'(a') = f_a(a',k')$ provided $f(\cdot,k')$ is locally semiconvex near $a'$;  similarly,
$g''(a') \ge f_{aa}(a',k')$ provided both functions
admit a second order Taylor expansion with respect to $a$ at $a'$.
\end{lemma}

\begin{definition}[Supermodular]
Given intervals $I,J \subset \R$,
a function $f: I\times J \longrightarrow \R$ is {\em weakly supermodular} if
\begin{equation}\label{supermodular}
f(a,k) + f(a',k') \ge f(a,k') + f(a',k)
\end{equation}
for all $1\le a <a' \in I$ and $1 \le k < k' \in J$.
It is {\em strictly} supermodular if, on the same domain,
the inequality \eqref{supermodular} remains strict. 
\end{definition}


\begin{remark}[Supermodular extensions]
It is elementary to check that a function $f$ which is weakly (or strictly) supermodular
on $A \times K$ and has an upper semicontinuous extension to $\bar A \times \bar K$
that is continuous and real-valued except perhaps at $(\bar a,\bar k)$,  is weakly (respectively strictly) supermodular on $\bar A \times \bar K$.
%
\end{remark}


Throughout we assume $\theta,\theta',N,N'$ and $\bar a = \bar k$ are positive parameters with 
$\max\{\theta,\theta'\}<1\le N$, and set $c \ge 0$ and $A =[0,\bar a[=K$. Unless otherwise noted,
the utilities $b_E,b_L \in C^1(\bar K)$ 
of education and labor have positive lower bounds $\ub b_{E/L}'$ and $\ub b_{E/L}''$ 
on their first two derivatives \eqref{utility bound 0}--\eqref{utility bound 2},
hence are strictly convex and increasing.  

\begin{lemma}[Structure of wage functions]
\label{L:integral form}
Let $v:K\longrightarrow \R$ 
be convex non-decreasing, with $v(\bar k) \ge \limsup_{k \to \bar k} v(k)$. 
Then $f(a,k) = v(z(a,k))$ will be weakly supermodular on $\bar A \times \bar K$,
and strictly supermodular unless the convexity of $v$ fails to be strict.

Set $z(a,k)=(1-\theta)a + \theta k$,
$\bet = b_E \circ z$ and $\blt(k',k) = b_L((1-\theta')k'+\theta'k)$ where 
$b_{E/L} \in C^1(\bar K)$ satisfy \eqref{utility bound 0}--\eqref{utility bound 2}.
%
Then the student payoff $u$ defined by \eqref{student wage}
is also 
convex non-decreasing 
on $K$ and 
satisfies $\frac{u'}{1-\theta} \ge \cm \ub b_E' + \inf_{k} v'(k)$ and
$\frac{u''}{(1-\theta)^2} \ge \cm \ub b_E'' +\inf_{k}  v''(k)$ pointwise a.e.

The worker, manager, and teacher wage functions $v_{w / m / t}$ 
defined
by \eqref{wage_w}--\eqref{infinity convention} and their maximum ${\bar v} := \max\{v_w,v_m,v_t\}$
are then monotone and convex on $\bar K$, real-valued on $K$, and satisfy
$\bar v' \ge \min\{(1-\theta')\ub b_L',N'\theta' \ub b_L',N\theta (c \ub b_E'+\inf_{k} v'(k))\}$ and
$\bar v'' \ge \min\{(1-\theta')^2 \ub b_L'',(\theta')^2N'\ub b_L'',N\theta^2(c\ub b_E''+ \inf_{k} v''(k))\}$
pointwise a.e.\ 

\end{lemma}

\begin{proof}
First note that convexity and monotonicity imply $v$ is continuous throughout $K=[0,\bar k[$.
Any convex $v \not\in C^2$ can be approximated by convex $v_i \in C^2$
locally uniformly on $]0,\bar k[$,  with $v_i' \to v'$ pointwise a.e. (and $v_i'' \to v''$ weakly).


Now let $f(a,k) = \cm b_E(z(a,k)) + v(z(a,k))$.  
For each fixed $\tilde k$, we see $f$ is convex non-decreasing as a function of $a \in \bar A$,
so the same must be true of the supremum 
$u(a) = \sup_{k \in \bar K} f(a,k) - v(k)/N$.  
Supposing for simplicity that $v$ and $b_E$ are $C^2(\bar A)$, from
$$\ts
f_{a}(a,k) = 
\left(\cm b_E'(z(a,k)) + v'(z(a,k))\right) z_a(a,k)
$$
and $0 \le z(a,k) = (1-\theta)a + \theta k$
we compute bounds
$$
\cm \ub b_E'+ \inf_{} v' 
\le \frac{f_{a}(a,k)}{1-\theta} \le 
\cm b_E'(z(a,\bar k)) + v'(z(a,\bar k)) 
$$
and
\begin{eqnarray*}
\frac{f_{aa}(a,k)}{(1-\theta)^2} &=& 
\cm b_E''(z(a,k)) + v''(z(a,k)) 
\\ &\ge& 
\cm \ub b_E''+ \inf v'' 
\end{eqnarray*}
which are uniform in $k \in \bar K$.
The
analogous bounds for $u$ follow from Lemma~\ref{L:inherit Lipschitz and semiconvex}.

So far,  we have been working under the assumption that $v$ and $b_E$ are $C^2(\bar K)$.
More generally,  $v$ and $b_E$ can be approximated uniformly on compact subsets of
$K$ by $C^2$ functions $v^i$ and $b_E^i$ satisfying the same hypotheses as $v$ and $b_E$.
As a result, $f^i(a,k) := b^i_E(z(a,k))  + v^i(z(a,k))$ converges to $f$ uniformly
on compact subsets of $\bar A^2 \setminus \{(\bar a,\bar k)\}$,
and $u^i(a) := \sup_{k \in \bar K} f^i(a,k) - \frac{1}{N} v(k)$ converges
uniformly to $u$ on compact subsets of $A$.  Thus $u$ inherits the same Lipschitz
and local semiconvexity bounds as $u^i$ in the distributional 
(and hence pointwise a.e.) sense. See \eqref{distributional} for the distributional definition
of the inequality $v_i'' \ge g$.

On the other hand,  $f(a,k; \theta) = f(k,a; 1-\theta)$ is symmetrical, 
and $v_t(k)/N$ is defined by
essentially the same formula as $u(a)$,  but with the roles of $a \leftrightarrow k$
and $\theta \leftrightarrow 1-\theta$
interchanged. Thus $v_t$ is also locally Lipschitz and convex
on $K$, and satisfies
$
v_t' \ge N\theta (\cm \ub b_E'+ \inf_b v'(b))$ and
$v_t'' \ge N\theta^2 (\cm \ub b_E''+ \inf_b v''(b))$.

Turning to $v_w$ and $v_m$,  we apply Lemma \ref{L:inherit Lipschitz and semiconvex}
but with $f(a,k) := \blt(a,k)= b_L((1-\theta')a + \theta' k)$,  which is jointly
convex and increasing in each variable. 
Approximating $b_L$ by $C^2(\bar K)$ functions if necessary, shows bounds  
$$\ub b_L'= b_L'(0^+) \le \frac{f_{a}(a,k)}{1-\theta'} = b_L'((1-\theta')a + \theta'k) 
\le b_L((1-\theta')a^- +\theta'{\bar k}) \le \bar b_L'
$$
and 
$$
\frac{f_{a a}(a,k)}{(1-\theta')^2} = b_L''(z'(a,k)) \ge \ub b_L'' 
$$
are inherited by the convex increasing functions $v_w$ and $\frac{1}{N'}v_m$ on $K$.
Thus ${\bar v} = \max\{v_w,v_m,v_t\}$ is convex, non-decreasing,
locally Lipschitz and
inherits the bounds
${\bar v}' \ge \ub \min\{(1-\theta')\ub b_L',N'\theta'\ub b_L',N\theta(\cm \ub b_E' + \inf_b v'(b))\}$
and ${\bar v}''  \ge \min\{(1-\theta')^2\ub b_L'',N'(\theta')^2\ub b_L'',N\theta^2(\cm \ub b_E''+ \inf_b v''(b))\}$
on $K$.

Finally,  setting $f(a,k) = v(z(a,k))$, using convexity of $v \in C^2(\bar K)$
we compute 
\begin{eqnarray*}
f(a_0,k_0)& +& f(a_1,k_1) - f(a_0, k_1) - f(a_1,k_0)
\\ &=& (1-\theta)\theta  \int_{a_0}^{a_1} \int_{k_0}^{k_1} v''((1-\theta)a+\theta k) da dk 
\\ &\ge&0
\end{eqnarray*}
for $a_0 < a_1$ and $k_0<k_1$.  For $v \not\in C^2$,  the same formulas hold by smooth 
approximation of $v = \lim v^i$.
Strict inequality holds unless $v'' = 0$ throughout $]z(a_0,k_0),z(a_1, k_1)[$. 
This yields the (strict) supermodularity \eqref{supermodular}
asserted. 
\end{proof}


We are now in a position to prove our first main result, which
describes how occupations are allocated according to cognitive skill.
It depends on the relative size of various parameters:
the teaching capacity $N$ (resp. $N'$) and effectiveness $\theta$  (resp. $\theta'$)
of teachers (resp. managers) in the population in question,
the range $\bar k$ of cognitive skills,
and the relative utility $\cm \ge 0$ of cognitive achievement
compared to wages.  When $N'\theta'$ and $cN\theta$ are large enough
it turns out that there is a complete ordering (a)-(b)
of skill types between workers, managers, and teachers in a steady-state economy.  
However 
$N\theta \ge 1$ is enough to ensure that 
no student studies with a teacher whose cognitive skills are inferior to their
own (d), while $N'\theta'$ and $N\theta$ large enough guarantee
that the most cognitively skilled types all become teachers (c)
(though not that all teachers have high cognitive skills).
This conclusion will help us  
to establish the phase transition 
from bounded to unbounded wage gradients that these teachers
enjoy as $N\theta$ passes through $1$ (in section \ref{S:phase transition}).
The possibility (f) that the number $d(k)$ of types of academic descendants a teacher can have
may grow without bound as $k \to \bar k$ foreshadows the analysis there. 

%

\begin{remark}
Note that 
in the following proposition, (c) and (d) 
together imply (e), meaning at least one
of the two inequalities $N\theta \ge 1$ or $\cm\ge 0$ is strict.
Also note $N'\theta'\ge \bar b'_L/ \ub b_L'$ and $N\theta \ge \bar b'_L/\ub b_L'$
are sufficient for (b) and (c), respectively.
\end{remark}

\begin{proposition}[Specialization by type; the educational pyramid]\label{P:specialization}
Fix $K=[0,\bar k[$ with $\bar k>0$,
and $\cm 
\ge 0$. Suppose $u,v:K \longrightarrow \R$ are
convex, nondecreasing, and 
satisfy $v=\max\{v_w,v_m,v_t\}$. 

If (a) $N\theta\cm  \ub b_E' \ge \bar b_L' \max\{N'\theta',1-\theta'\}$
then all teacher types lie weakly above all of the manager and worker types.

If (b) 
$N'\theta' > (1-\theta') \sup_{k\in K} b_L'(1-\theta')k+\theta'\bar k^-)/b_L'(\theta'k^+)$ 
then all of the worker types lie weakly below all of the manager types.

If (c) $N\theta \ge \sup_{0\le z \le k} b'_L((1-\theta')z^-+\theta'\bar k)/ (b_L'(\theta'z^+)+\frac c{N'\theta'}b_E'(z^+))$ 
and (b) holds, and
$f(a,k) := u(a) + \frac 1N v(k) -c b_E(z(a,k)) - v(z(a,k))$ vanishes at some
$(a,k) \in K \times K$ where $v(z(a,k)) = v_m(z(a,k))$, then $v> v_m$ on $]k,\bar k]$. 
In other words,  no manager (or worker) can have a type higher than a teacher of managers.

If (d) $N\theta \ge 1$, 
then any student of type $a\in K$ will be weakly less skilled
than his teacher, and strictly less skilled if (e) either $c>0$ or $N\theta>1$ in addition.

If (f) either $\cm>0$ or $v'(0^+)>0$, 
then (d)--(e) imply all academic descendants of a teacher with skill $k \in K$
will display one of at most finitely many $d=d(k)$ distinct skill types, unless 
differentiability of $v$ fails at $k$.  However, $d(k)$ may
diverge as $k \to \bar k$, in which case $v'(k)\to+\infty$ at a rate related
to $d(k)$ by \eqref{diverging sum}. 

\end{proposition}

%
%

\begin{proof}
Lemma \ref{L:integral form} asserts convexity of $v_{w/m/t}$,
hence one-sided differentiability everywhere, and two sided-differentiability
except perhaps at countably many points. 
At points $k\in]0,\bar k[$ of differentiability,
Lemma \ref{L:inherit Lipschitz and semiconvex}
(the envelope theorem)
allow us to
estimate the 
wage gradients
\begin{eqnarray}
\ts v'_w(k)
=& (1-\theta') b_L'(z'(k,k'_m))
&\in (1-\theta')]\ub b_L',\bar b_L'[
\label{estimate_w}
\\   
\ts v'_m(k)
=& N'\theta' b_L'(z'(k_w',k))
&\in \phantom{(-} N'\theta']\ub b_L',\bar b_L'[,
\label{estimate_m}
\\ 
\ts v_t'\,(k)
=& N \theta(\cm b_E'(z(a,k))+ v'(z(a,k)) )
& \ge \phantom{(-} N\theta c b_E'(\theta k)
\label{estimate_t}
\end{eqnarray}
where $k_m',k_w'$ and $a$ are the respective points at which the suprema
\eqref{wage_w}--\eqref{wage_t} (or their extension to $\bar K$) are attained.
Such points exist in $\bar K$ according to the same lemma; we can extend
$v\left(\bar k\right) = v\left(\bar k^-\right) :=\lim_{k \uparrow \bar k} v(k)$
and $u\left(\bar k\right) \in \R \cup \{+\infty\}$ similarly
without changing $v_{w/m/t}$.
Consideration of the worst-case scenario $k'_w =0$ and $k'_m = \bar k$ 
in \eqref{estimate_w}--\eqref{estimate_m} shows 
if (b) holds 
that
$v_m'(k) < v_w'(k)$ at each point $k$ where both derivatives are defined.
Then the locally Lipschitz function $v_m-v_w$ is strictly increasing.
Since this function is non-positive on $\{k \mid \bar v=v_w\}$ and non-negative
on $\{k' \mid \bar v=v_m\}$, the first set must lie
entirely to the left of the second,  as desired.

Estimating the wage gradient for a teacher of type $k_0=k \in K$ is more subtle,
due to the recursive nature of formula \eqref{estimate_t}.  Since the student
of ability $a_1=a$ taught by $k_0$ winds up with cognitive skill
$k_1 = (1-\theta)a_1+\theta k_0 =z(a_1, k_0)$,  we find
\begin{eqnarray}\label{iterate 0}
v_t'(k_i)
&=& N\theta (\cm b_E'(k_{i+1}) + v'(k_{i+1}))
\end{eqnarray}
for $i=0$, assuming differentiability of $v_t$ at $k_0$.  Differentiability of $v$ and $b_E$
at $k_1$ (and also of $v_t \le v$) follows
from convexity, since replacing $k$ by $k_0$ produces equality in
$u(a_1) + \frac{1}{N}v_t(k) - v(z(a_1,k)) - \cm b_E(z(k,a_1)) \ge 0$: the first-order condition
\begin{eqnarray*}
(v'+\cm b_E') ({z(a_1,k_0)^-}) \ge & \frac{1}{N}v_t'(k_0)/z_k(a_1,k_0) & \ge 
(v'+ \cm b_E')(z(a_1,k_0)^+)
\end{eqnarray*}
forces the one-sided derivatives $(v'+cb_E')(k_1^-) \le (v'+c b_E')(k_1^+)$ to agree.
From \eqref{iterate 0} we have
$v_t'(k_0) \ge N\theta\cm \ub b_E'$,
which dominates $(1-\theta') \bar b_L'$ and $N'\theta' \bar b_L'$ in case (a).
Since $v_{w/m/t}$ are locally Lipschitz,
monotonicity of $v'(k)$ then combines with
the estimates \eqref{estimate_w}--\eqref{estimate_m} already established to
show all teacher types $k_0$ are at least as high as the highest worker and manager types.

From (d) $N\theta \ge 1$ and \eqref{iterate 0} we conclude $v'(k_1) \le v_t'(k_0)$,
and this inequality is strict if (e) also holds, in which case
every student studies with a teacher more skilled than himself, or --- what is equivalent
in our model --- no student (except the very top type $a=\bar a$)
becomes as skilled as his teacher.

Next, assume as in case (c), that a teacher of type $k \in K$ teaches
a student of type $a$ who becomes a manager of type 
$z = z(a,k)$.  Since $v \ge v_m$ with equality at $z$, we have 
$v'(z^+) \ge v_m'(z^+)$.
 Analogously to \eqref{estimate_m}--\eqref{estimate_t} we find
\begin{eqnarray*}
\frac{1}{N\theta}v_t'(k^+) &\ge& 
c b_E'(z^+) + v_m'(z^+)
\\ &\ge & cb_E'(z^+) + N'\theta' b_L'({(1-\theta')k_w + \theta' z^+}),
\\&\ge & cb_E'(z^+) + N'\theta' b_L'({\theta' z^+}) 
\end{eqnarray*}
with equality holding in the first two estimates if all the derivatives in question exist.
On the other hand, 
\begin{eqnarray*}
v_m'(\bar k) & \le & N'\theta' b_L'( (1-\theta')z^- + \theta'\bar k)
\end{eqnarray*}
since (b) implies the worker types all lie below the manager type $z$.
Hypothesis (c) now yields $v_t'(k^+) \ge v_m'(\bar k)$, and the convexity
of $v_t$ and strict convexity of $v_m$ shown as in Lemma \ref{L:integral form}
then imply $v_t' > v_m'$ on $]k,\bar k[$.  Vanishing of the non-negative function $f$
at $(a,k)$ implies $v_t(k) = v(k) \ge v_m(k)$, whence the desired conclusion 
$v_t> v_m$ follows on $]k,\bar k]$ by integration.



Case (f) is more delicate, and our conclusions for it are more involved.
If the student $a_1$ above elects to become a worker or manager, we can estimate
\eqref{iterate 0} using \eqref{estimate_w}--\eqref{estimate_m}.
However,  if the student becomes a teacher whose students' innate ability $a_2$ allows
them to acquire human capital $k_2 = z(a_2,k_1)$, we must iterate \eqref{iterate 0}.
And if these students in turn become teachers teaching students of ability $a_3$ to
acquire human capital $k_3 = z(a_3,k_2)$, we must iterate again,  and continue
iterating until the student of ability $a_d$ who acquires human capital
$k_d =z(a_d, k_{d-1})$ elects 
to become a worker or manager instead of another teacher.
Assuming 
(d)--(f),
we claim this occurs for some finite $d$:  otherwise
the skills $k_{i+1} < k_i$ converge to some $k_\infty \in K$, for which
the limit of \eqref{iterate 0} yields an identity 
$(\frac{1}{N\theta}-1)v'(k_\infty^+) =  \cm b_E'(k_\infty^+)$
equating quantities with different signs.
Recalling $v'(k_\infty^+) \ge 0$ and $\cm \ge 0$, hypothesis (f) asserts at
least one of these inequalities is strict,  while (d) asserts $N\theta \ge 1$.
Unless $N\theta=1$ and $\cm=0$,  this contradicts the limiting identity.
But $N\theta=1$ and $\cm=0$ contradicts (e).  Thus the
sequence $k_i$ terminates at some finite $d$ (which depends on $k_0$).

At this point we have
\begin{eqnarray}
\ts
v_t'(k)
&=& N\theta \Bigg(\cm b_E'(k_1) + N\theta \bigg(\cm b_E'(k_2) + N\theta\Big(...
+ N\theta\left(\cm b_E'(k_d) + v'(k_d)\right)
\Big)\bigg)\Bigg)
\cr&\ge& \left\{
\begin{array}{ll}
 \ds \frac{1-(N\theta)^d}{1-N\theta} 
N\theta c b_E'(\theta^dk) + \ts (N\theta)^{d}v'(\theta^d k) 
  & {\rm if}\ N\theta\ne 1\\
d \cm b_E'(\theta^d k) + v'(\theta^d k)  & {\rm if}\ N\theta=1,
\end{array}
\right.
\label{diverging sum}
\end{eqnarray}
where we have summed the geometric series and
estimated 
 $k_d \ge \theta^{d} k_0$.
%
\end{proof}

\subsection{Characterization of optimality}

When we turn to the question of existence of optimal payoffs $(u,v)$ for the
linear program \eqref{three line dual},  our strategy will be to perform
the minimization under the additional
assumption that $u$ and $v$ are convex non-decreasing,
 and then to show these additional constraints are non-binding at the optimum,  thus
 have no effect on the outcome.  Convexity and monotonicity provide the
requisite compactness for extracting limits from minimizing sequences.
In order to show these constraints are non-binding however, it is necessary to
 control the payoff $u(a)$ on the full interval $A=[0,\bar a[$,
 and not only on $\spt \alpha$.  Similarly it is necessary to control $v$ on the
  full interval $K=A$,  and not only on the support of the unknown distribution $\kappa$
  of adult skills. Since the original problem is largely insensitive to the values of
  $u$ and $v$ outside $\spt \alpha$ and $\spt \kappa$, 
we introduce a perturbed version of the problem to provide this control:
for each $\delta>0$ set
\begin{equation}\label{delta dual}
LP(\delta)_* :=
\inf_{( u, v) \in F_\delta} 
\delta \langle u + v \rangle_A + \int_{[0,\bar a]} u(a) \alpha(da)
\end{equation}
where $\langle v \rangle_A := \frac1{H^1(A)}\int_A v dH^1$ 
denotes the Lebesgue average
of $v$ over $A$.
Here $F_\delta=F_0$  denotes the same feasible set as before,
with a subscript 
denoting only the possible dependence of the constant
$c=c_\delta$ in \eqref{stability education} on $\delta>0$.
Also $u$ (and hence $v$) $ \in L^1(\bar A,\alpha)$, and if $\delta>0$ then
$u,v \in L^1(A,H^1)$. 
We must first solve the perturbed problem \eqref{delta dual}
and then extract the $\delta \to 0$ limit.  For the latter
endeavor  and to characterize the optimizers,
it will be crucial to know $LP(\delta)_*$ is in fact dual to 
\begin{equation}\label{delta primal}
LP(\delta)^*:=\max_{
\epsilon,\lambda \ge 0\ {\rm on}\ \bar A \times \bar K\ {\rm satisfying}\ \eqref{delta student marginal}-\eqref{delta steady state}
}
\cmd \epsilon(b_E \circ z) + \lambda(b_L \circ z')
\end{equation}
where 
\begin{equation}
\label{delta student marginal}
\epsilon^1 
= \alpha + \frac{\delta}{|A|} H^1|_{A} 
\end{equation}
and
\begin{equation}
\label{delta steady state}
\lambda^{1} + \frac{1}{N'} \lambda^{2} + \frac{1}{N} \epsilon^{2} 
= z_\#\epsilon + \frac{\delta}{|K|} H^1|_K.
\end{equation}

Let us begin by verifying $LP(\delta)^*\le LP(\delta)_*$.
This would be standard if the primal infimum 
were restricted to continuous bounded functions $u,v \in C(\bar A)$,
as in Appendix \ref{S:no duality gap} where the reverse inequality
and attainment of the dual maximum 
are verified.  
However, a priori we know only that $u,v$ differ from continuous bounded
functions by non-decreasing functions, and even a posteriori we do not know whether or not minimizers
of \eqref{three line dual} or \eqref{delta dual} are bounded at $\bar k$.
We have only the conditional result of Theorem \ref{T:phase transition}
to suggest that they are.  Thus we are forced to work in a space which
includes unbounded functions, and to check their inclusion does not
spoil the otherwise elementary duality inequality $LP(\delta)^* \le LP(\delta)_*$.


\begin{proposition}[Easy direction of duality for unbounded functions]
\label{P:easy duality}
Fix $\delta,\cmd$ non-negative and $\theta,\theta',N,N',\bar a=\bar k$ positive with 
$\max\{\theta,\theta'\}<1\le N$.
Let  $\alpha$ be a Borel probability measure on $\bar A$,
where $A=[0, \bar a[=K$, and define 
$z(a,k) =(1-\theta)a + \theta k$, $\bet = b_E \circ z$ and 
$\blt(a,k)=b_L((1-\theta')a+\theta'k)$ where 
$b_{E/L}\in C^0(\bar K)$. 
If Borel measures $(\epsilon,\lambda)$ and Borel functions
$(u,v)\in F_\delta$ are feasible for the primal and dual problems
\eqref{delta dual}--\eqref{delta primal},
with $u \in L^1(\bar A,\alpha)$ and $u\delta, v\delta \in L^1(A, H^1)$,
then 
$
\alpha(u) + \delta \langle u+v\rangle_A \ge \cmd \epsilon(\bet) + \lambda(\blt)
$ 
provided $v \in L^1(\bar A,z_\#\epsilon)$.
If $\alpha$ satisfies the doubling condition \eqref{doubling},
then $v \in L^1(\bar A,z_\#\epsilon)$.
\end{proposition}

\begin{proof}
Taking feasible pairs $(\epsilon,\lambda)$ of measures and $(u,v) \in F_\delta$ of functions
with $u \in L^1(\bar A,\alpha)$ and $u\delta,v\delta \in L^1(A,H^1)$,
the stability constraint for the education sector implies
\begin{equation}\label{L1 upper bound}
u(a) - \cmd b_E(z(a,k)) \ge \ts  v(z(a,k))-\frac{1}{N}v(k),
\end{equation}
on $\bar A \times \bar K$, and the left hand side is in $L^1(\bar A^2,\epsilon)$.
Thus
\begin{eqnarray}
\nonumber
+\infty &>& \alpha(u) - \cmd\epsilon(\bet) + \delta \langle u+v \rangle_{A}
\\ &\ge& \langle \delta v \rangle_{K} + \int_{\bar A \times \bar K}
[v(z(a,k))-\ts \frac1N v(k) ] \epsilon(da,dk)
\label{upper bound}
\end{eqnarray}
since $\epsilon^1=\alpha + \frac\delta{|A|} H^1|_A$. On the other hand, the steady state constraint
$z_\#\epsilon + \frac\delta{|K|} H^1|_K= \lambda^1 + \frac1{N'} \lambda^2 + \frac1N \epsilon^2$
combines with the stability constraint $v(a)+ \frac1{N'} v(k) \ge \blt(a,k)$
for the labor sector to imply
\begin{eqnarray}
\label{lower bound}
\langle \delta v \rangle_K + \int_{\bar K} v d(z_\#\epsilon - \ts\frac{1}{N} \epsilon^2\ds)
&=& \int_{\bar K} v d(\lambda^1 + \ts\frac{1}{N'} \lambda^2)\ds
\\ &\ge& \int_{\bar A \times \bar K} \blt d\lambda
\label{next bound}
\\ \nonumber &>& 0.
\end{eqnarray}
Now if $v \in L^1(\bar A, z_\#\epsilon)$  we can equate the right hand side of
\eqref{upper bound} with the left hand side of \eqref{lower bound} 
to obtain the first stated conclusion. 

We must still show that the doubling \eqref{doubling} of $\alpha$ at $\bar a$
implies $0 \le v \in L^1(\bar A,z_\#\epsilon)$.
Recall that $(u,v) = (u_0+u_1,v_0+v_1)$ with $u_0,v_0 \in C(\bar A)$ and 
$u_1,v_1:\bar A \longrightarrow [0,\infty]$ non-decreasing (in fact strictly increasing
without loss of generality). Since $v_0$ is bounded there is no question about its
integrability.  We shall use $v \le u$ from \eqref{stability bound} and $u \in L^1(\bar A,\alpha)$ 
to deduce $v_1 \in L^1(\bar A,\kappa)$ for $\kappa := z_\#\epsilon$.  
Since $v_1$ is strictly increasing,  $v_1^{-1}(y) \in \R \cup\{\pm \infty\}$ can be defined unambiguously.
Lemma \ref{L:adult tail bound}, the doubling condition \eqref{doubling},
and the layer-cake representation \cite{LiebLoss97} of the Lebesgue integral imply
\begin{eqnarray}
\int_{\bar K} 
v_1(k) \kappa(dk) 
&=& \int_0 
^\infty \kappa[ v_1^{-1}[y,\infty]]
dy\cr
&=& \int_0^\infty \kappa[\bar a - (\bar a -v_1^{-1}(y)),\bar a] dy\cr
&\le& \int_0^\infty \alpha[\bar k - \frac{1}{1-\theta}(\bar a - v_1^{-1}(y)),\bar a] dy \cr
&\le& C^{\frac{1}{\theta}-1} \int_0^\infty \alpha[v_1^{-1}(y),\bar a] dy
\label{kappa(v) finite}
\end{eqnarray}
for some $C<\infty$.  On the other hand,
$v_1 \le u_0 + u_1 - v_0 \le u_1+ const$ 
yields $u_1^{-1}(y-const) \le v_1^{-1}(y)$,
so 
$$
 \int_{0}^\infty \alpha[u_1^{-1}(y),\bar a] dy = \int u_1(a) \alpha (da) < +\infty
 $$
implies finiteness of \eqref{kappa(v) finite} and completes the proof that $v_1 \in L^1(\bar K,\kappa)$.
%
\end{proof}

\begin{corollary}[Characterizations of optimality]\label{C:characterize}
Fix $\delta,\cmd,\theta,\theta',N,N',z,\bet,\blt$ and $\bar k = \bar a$ as in Proposition
\ref{P:easy duality},
and a Borel probability measure $\alpha$ on $\bar A$ satisfying \eqref{doubling},
where $A=[0, \bar a[=K$ and $\bar a \in \spt \alpha$.
A pair of feasible measures $\epsilon,\lambda \ge 0$ on $\bar A^2$ maximizes the dual
problem \eqref{delta primal} if there exist feasible $(u,v) \in F_\delta$
such that $\alpha(u) + \delta \langle u+v \rangle_A = \cmd \epsilon(\bet) + \lambda(\blt)$.

Conversely,  $(u,v) \in F_\delta$ 
minimize the primal problem
if and only if there exist $\epsilon,\lambda \ge 0$ feasible for the dual
problem
such that $\alpha(u) + \delta \langle u+v\rangle_A= \cmd\epsilon(\bet) + \lambda(\blt)$.

For feasible pairs in the given spaces,
$\alpha(u) + \langle u+v\rangle_A= \cmd \epsilon(\bet) + \lambda(\blt)$
is equivalent to the assertions $\epsilon(f)=0=\lambda(g)$
where $f(a,k) = u(a) + \frac{v(k)}{N} - \cmd \bet(a,k) - v(z(a,k)) \ge 0$
and $g(k',k) = v(k') + \frac{v(k)}{N} - \blt(k',k) \ge 0$ on $\bar A \times \bar K$.
\end{corollary}

\begin{proof}
Let $(\epsilon,\lambda)$ be a pair of feasible measures for the
dual problem,  and $(u,v) \in F_\delta$ 
so that $u \in L^1(\bar A,\alpha)$ and $u\delta,v\delta \in L^1(A,H^1)$
and $f,g \ge 0$ when defined as above.
Then Proposition \ref{P:easy duality} asserts $v \in L^1(\bar A, z_\#\epsilon)$ and
$\cmd\epsilon(\bet) + \lambda(\blt) \le LP(\delta)^* \le LP(\delta)_*
 \le \alpha(u) +\delta \langle u + v\rangle_A$.
If $\cmd \epsilon(\bet) + \lambda(\blt) = \alpha(u) + \delta \langle u + v\rangle_A$ this
forces this chain of inequalities to become equalities,
showing $(\epsilon,\lambda)$ and $(u,v)$ to optimize their
respective problems.

The converse is proved using the result $LP(\delta)^* = LP(\delta)_*$,
which follows by combining the same proposition with Theorem~\ref{T:no gap}.
Suppose $\alpha(u) + \delta \langle u + v\rangle_A =LP(\delta)_*$,  meaning
$(u,v)$ is optimal.  Lemma \ref{L:primal existence} provides
$(\epsilon,\lambda)$ such that $\cmd \epsilon(\bet) + \lambda(\blt)=LP(\delta)^*$.


Finally, we claim that
$\cmd \epsilon(\bet) + \lambda(\blt)=\alpha(u) + \delta \langle u + v\rangle_A$
is equivalent
to $\epsilon(f)=0=\lambda(g)$.
This follows from the chain of inequalities which establish
$c\epsilon(\bet) + \lambda(\blt) \ge \alpha(u)$ in Proposition \ref{P:easy duality}:
$\epsilon(f)=0$ is equivalent to equality in \eqref{upper bound},
$\lambda(g)=0$ is equivalent to equality in \eqref{next bound}, and
when both of these hold then
$\cmd\epsilon(\bet) + \lambda(\blt)=\alpha(u) + \delta \langle u + v\rangle_A$.
\end{proof}

\begin{remark}[Converse]\label{R:characterize}
According to Theorems \ref{T:minimizing wages} and \ref{T:no gap}, the sufficient condition
for optimality of $(\epsilon,\lambda)$ given by Corollary \ref{C:characterize}
is also necessary.
\end{remark}

\subsection{Optimal wages for the primal problem}

Using the foundations laid in the previous sections,  we are
ready to demonstrate the existence of optimal wages $v(k)$
and payoffs $u(a)$ for the primal problem \eqref{three line dual}.
This is done using a compactness and (lower semi-)continuity argument
for the 
perturbed problem \eqref{delta dual},  and then
taking the limit $\delta \to 0$.  For $\delta>0$, we assume $v$
is convex  nondecreasing,
and then use Lemma \ref{L:integral form} and the characterization
$v=\max\{v_w,v_w,v_t\}$ --- which identifies the wage of an ability $k$ adult
with the maximum he can earn as a worker, manager or teacher ---
to show the convexity and monotonicity
assumptions on $v$ do not bind,  so play no role in the outcome of our
(infinite-dimensional) linear program.
Thus convexity of the wages in our model emerges
for reasons which manifest rather differently 
than in Rosen's investigation of superstars~\cite{Rosen81}.

Compactness for convex non-decreasing $v$ is asserted in the
following lemma.  Some delicacy is required to show that if $v$ or $u$
diverges to $+\infty$,  then both do so on the same half-open interval,
and at a uniform rate.


\begin{lemma}[Compactness for wage functions]\label{L:compactness}
Fix $K=[0,\bar k[$ and $g \in L^1_{loc}(K)$.
A sequence $v_i:K \longrightarrow [0,\infty[$ of convex non-decreasing
functions 
satisfying $v_i''(k) 
\ge g(k)$ a.e.,
admits a subsequence which converges pointwise to a limit
$v_0 : K \longrightarrow [0,\infty]$ which is
real valued on $[0, k_0[$, and infinite on $]k_0,\bar k[$, for some
$k_0 \in [0,\bar k]$.  The convergence is uniform on compact subsets
of $[0,k_0[$, and the analogous bound 
$v_0''(k) \ge g(k)$ holds a.e.\ on 
its interior.  Furthermore,
 for $a> k_0$
$$
u_i(a) := \max_{k \in [0, \bar k[}
\ts \cm b_E(z(a,k)) + v_i(z(a,k)) - \frac{1}{N} v_i(k)
$$
diverges to $u_0(a) = +\infty$ as $i \to \infty$ along the subsequence
described above, where $b_E \in C^1(\bar K)$ satisfies 
\eqref{utility bound 0}-\eqref{utility bound 2}, $c\ge 0$ and $z(a,k) = (1-\theta)a + \theta k$.
\end{lemma}

\begin{proof}
The fundamental theorem of calculus yields
\begin{equation}\label{i form}
v_i(k') = v_i(0) + \int_{0}^{k'} v_i'(k) dk.
\end{equation}
Since $0 \le v_i'(k)$ is non-decreasing for each $i$,
Helly's selection theorem provides
a subsequence converging to a non-decreasing limit $v_0'(k)$ on $K$,  except possibly
at discontinuities of $v_0'$ in $]0,\bar k[$.
Choose a further subsequence for which
$v_0(0) := \lim_{j \to \infty} v_{i(j)}(0)$ converges;
unless such a sequence exists, $v_0(0) = +\infty$ and the lemma
follows immediately with $k_0 = 0$.  Therefore assume $v_0(0) <\infty$
and choose $k_0\in [0,\bar k]$ so that
$v_0'(k) <\infty$ for $k< k_0$ and $v_0'(k)=\infty$ for $k>k_0$.
For $k'<k_0$,  Lebesgue's dominated convergence theorem allows us to
take $i(j) \to \infty$ in \eqref{i form},  to obtain a continuous limit $v_0(k')$
on $[0, k_0[$.  It follows that $v_{i(j)} \to v_0$ uniformly on compact subsets
of $[0, k_0[$.  Monotonicity of $v_i'$ ensures $v_{i(j)}(k) \to \infty$ for each
$k> k_0$.  For $v_i,g \in L^1_{loc}$,  the inequality
$v_i'' \ge g$ holds in the distributional sense --- meaning
\begin{equation}\label{distributional}
\int_{0}^{\bar k} [f''(k)v_i(k) - f(k) g(k) ]dk \ge 0
\end{equation}
for each smooth compactly supported test function $0 \le f \in C^\infty_c(]0,\bar k[)$
--- if and only if it holds in
the a.e.\ sense.  Thus $v_i'' \ge g$ distributionally, and  the bound
$v_0'' \ge g$ follows on $]0, k_0[$,  using Lebesgue's dominated convergence
theorem again.  Taking $g=0$ shows $v_0$ is convex on $]0,k_0[$, for example.

Now if $a > k_0$,  taking $k=k_0$ implies $k_0<z(a,k)=(1-\theta)a+\theta k$,  thus
$u_{i(j)}(a) \ge \cm v_{i(j)}(z(a,k_0))$ diverges to $+\infty$ as $j \to \infty$.
\end{proof}

\begin{corollary}[Convergence uniform from below]\label{C:uniform from below}
Suppose a sequence $v_i:[0, \bar k[ \longrightarrow [0,\infty[$
of functions satisfying the hypotheses of Lemma \ref{L:compactness}
converges pointwise to $v_0 : [0, \bar k[ \longrightarrow [0,\infty]$ which is
real valued on $[0, k_0[$, and infinite on $]k_0,\bar k[$ for some $k_0 \in [0,\bar k]$.
If $v_0(k_0^-):= \lim_{k \uparrow k_0} v_0(k)<+\infty$ then
\begin{equation}\label{one-sided}
0 \le \liminf_{i\to \infty} \inf_{k \in [0,k_0[} v_i(k) - v_0(k).
\end{equation}
On the other hand,  if $v_0(k_0^-) = +\infty$ then the sequence grows uniformly in the
sense that for each $c<\infty$ taking $i'<\infty$ large enough implies
$v_i(k) \ge c$ for all $k>k_0 -1/i'$ and $i>i'$.
\end{corollary}

\begin{proof}
Given $\delta>0$,  taking $k_1<k_0$ sufficiently large makes $v_0(k_1)> v_0(k_0^-)-\delta/2$.
Taking $i$ sufficiently large then ensures $v_i(k_1)>  v_0(k_0^-)-\delta$,
whence for all $k \in [k_1,k_0[$ monotonicity yields $v_i(k) > v_0(k) -\delta$.
Since the convergence $v_i \to v_0$ is uniform on $[0,k_1]$,
this concludes the corollary in case $v_0(k_0^-)<+\infty$ is finite.

If $v_0(k_0^-)=+\infty$,  given $c<\infty$ take $i'$ sufficiently large that
$v_0(k_0-1/i')>c$ and then larger still to ensure $v_i(k_0-1/i')>c$ for all $i>i'$.
Monotonicity again concludes the proof.
\end{proof}




\begin{theorem}[Existence of minimizing wages]\label{T:minimizing wages}
Fix $\cm\ge0$ and positive $\theta,\theta', N,N'$ and $\bar a=\bar k$ 
with $\max\{\theta,\theta'\}<1\le N$.  Set  $A=[0, \bar a[=K$ and
let  $\alpha$ be a Borel probability measure on $\bar A$ satisfying the doubling
condition \eqref{doubling} at $\bar a \in \spt \alpha$. Define
$z(a,k) =(1-\theta)a + \theta k$, $\bet=b_E\circ z$ and $\blt(a,k)=b_L((1-\theta')a+\theta'k)$,
where $b_{E/L} \in C^1(\bar K)$ satisfy \eqref{utility bound 0}--\eqref{utility bound 2}.
Then infimum \eqref{three line dual} is attained by functions
$(u,v)$ satisfying
$v = \max\{v_w, v_m, v_t\}$ on $\bar K=[0,\bar k]$ and
\begin{equation}\label{student wage}
u(a) = \sup_{k \in \bar K} \ts \cm b_E(z(a,k)) + v(z(a,k))- \frac{1}{N} v(k)
\end{equation}
on $\bar A$, 
where the $v_{w/m/t}$ are defined
by \eqref{wage_w}--\eqref{infinity convention};
here $u,v:\bar A \longrightarrow ]0,\infty]$ are continuous, convex, non-decreasing,
and --- except perhaps at $\bar a$ --- real-valued.
For 
$j\in \{1,2\}$,
 if $N\theta^j \ge 1$ 
 then $d^jv/dk^j \ge \ub b_L^{(j)}\min\{(1-\theta')^j,(\theta')^jN'\}$. 

\end{theorem}

\begin{proof}
Fix $0<\delta<1$ and $\cmd:=\cm>0$ positive;  if we prefer $c =0$
set $c_\delta = \delta$  in the $\delta \to 0$ limit procedure
which follows.
We are going to study the perturbed primal problem \eqref{delta dual}
under the same feasibility constraints \eqref{stability education}--\eqref{stability bound}
as \eqref{three line dual} ---
which include $u \in L^1(\bar A,\alpha)$ ---
plus the artificial constraint that
$v$ be convex nondecreasing. 
From \eqref{stability bound}, both $u$ and $v \in L^1(\bar A,\alpha)$  
and have positive lower bounds.
For $\delta>0$ we assume $u,v \in L^1(A,H^1)$ without loss
of generality,  since otherwise the term $\langle u+v \rangle_A =+\infty$
makes the objective diverge.
Feasibility of the pair
$(u,v)=(1+c_\delta\bar b_E/\bar b_L, 1)\bar b_L$ yields an upper bound
$(1+2\delta)(c_\delta \bar b_E + \bar b_L)$ for the infimum \eqref{delta dual}.
As remarked after \eqref{stability bound}, we may always replace $u$ and $v$
by their lower semi-continuous hulls without violating feasibility.  
Since $\alpha \ge 0$,  this only improves the objective \eqref{delta dual};
for the same reason, it costs no generality to henceforth suppose $u$ to be related to $v$
by \eqref{student wage}. Lemma~\ref{L:integral form} then implies both
$v$ and $u$ are convex and non-decreasing, hence continuous as
extended real-valued functions.

Lemma \ref{L:compactness} allows us to extract a subsequential limit
$(u_\delta,v_\delta)$ satisfying the same constraints from any sequence
of approximate minimizers for \eqref{delta dual}.
Fatou's lemma ensures the limit $(u_\delta,v_\delta)$ minimizes
the objective subject to these constraints.
Replacing the monotone convex functions $u_\delta$ and $v_\delta$ again by their lower-semicontinuous hulls
ensures both are continuous.
Since $\bar a \in \spt \alpha$,
our a priori bound $(1+2\delta)(c_\delta \bar b_E + \bar b_L)$ on  the objective implies the non-decreasing functions
$u_\delta(a)$ and $v_\delta(k)$ are finite, except possibly at $\bar a$ and $\bar k$,
and
\begin{equation}\label{a priori bound}
\int_{\bar A} u_\delta(a) \alpha(da) \le (1 + 2 \delta) (c_\delta \bar b_E + \bar b_L).
\end{equation}
Notice equality must hold in
\begin{equation}\label{perturbed student wage}
u_\delta(a) \ge \sup_{k \in [0,\bar k]} \ts \cmd b_E(z(a,k)) + v_\delta(z(a,k))- \frac{1}{N} v_\delta(k)
\end{equation}
since otherwise replacing $u_\delta$ by the right-hand side of
\eqref{perturbed student wage}
yields a feasible pair which lowers the objective functional, contradicting
the asserted optimality.
Use $(u,v)=(u_\delta,v_\delta)$ to define $(v_\delta^w,v_\delta^m,v_\delta^t):=(v_w,v_m,v_t)$ and
$\bar v_\delta := \max\{v_w,v_m,v_t\}$.

Feasibility implies $v_\delta \ge \bar v_\delta$, and Lemma \ref{L:integral form}
implies $\bar v_\delta$ is continuous on $K$, convex increasing on $\bar K$, and satisfies
\begin{eqnarray}
\label{V' bound}
\bar v_\delta'&\ge& 
\min\{(1-\theta')\ub b_L',N'\theta'\ub b_L',(\cmd \ub b_E'+ \inf v_\delta'(k))N\theta\} \quad {\rm and}
\\  \bar v_\delta'' &\ge& 
\min\{(1-\theta')^2 \ub b_L'',(\theta')^2N'\ub b_L'', (\cmd \ub b_E''+ \inf v_\delta'')N\theta^2\}
\label{V'' bound}
\end{eqnarray}
on $]0,\bar k[$. 
If $\eta := v_\delta-\bar v_\delta$ is positive somewhere,  it is positive on
an interval where the only binding constraints can be $v_\delta'=0$ or
$v_\delta''= 0$. For small $\lambda>0$, 
the perturbation $v^\lambda := (1 -\lambda)v_\delta + \lambda \bar v_\delta$
respects these differential constraints.  We will now show the pair $(u_\delta,v^\lambda)$ respects the
other constraints as well;  unless the continuous function $\eta=0$ throughout $K$, this pair
lowers the objective functional, a contradiction forcing $v_\delta = \bar v_\delta$.

Since $v^\lambda = v_\delta - \lambda\eta = \bar v_\delta + (1-\lambda) \eta$,
for $k',k \in \bar K$ we find
\begin{eqnarray}\nonumber
v^\lambda(k') + \ts\frac{v^\lambda(k)}{N'} - \blt(k',k)
&=& \ts \bar v_\delta(k') + \frac{v_\delta(k)}{N'} - \blt(k',k) + (1 - \lambda) \eta(k') -  \frac{\lambda}{N}\eta(k) \\
&\ge& \ts \eta(k')[1 - \lambda (1 + \frac{1}{N} \frac{\eta(k)}{\eta(k')})],
\label{first estimate}
\end{eqnarray}
and also
\begin{eqnarray}\nonumber
v^\lambda(k') + \ts\frac{v^\lambda(k)}{N'} - \blt(k',k)
&=& \ts v_\delta(k') + \frac{\bar v_\delta(k)}{N'} - \blt(k',k) - \lambda \eta(k') +  \frac{1-\lambda}{N}\eta(k) \\
&\ge& \ts \frac{\eta(k)}{N}[1 - \lambda (1 + \frac{N\eta(k')}{\eta(k)})].
\label{second estimate}
\end{eqnarray}
If both $\eta(k') \ge 0$ and $\eta(k)\ge 0$ are non-zero,  then taking $\lambda < 1/2$
ensures either \eqref{first estimate} or \eqref{second estimate} is positive.
The same conclusion remains true if one of $\eta(k')$ or $\eta(k)$ vanishes.
If both vanish,  there is nothing to prove.

On the other hand, adding $u_\delta(a) \ts - \cmd b_E(z(a,k))$
to
$$
\ts \frac{v^\lambda(k)}{N} - v^\lambda(z(a,k))
= \frac{\bar v_\delta(k)}{N} - v_\delta(z(a,k))
+\frac{1-\lambda}{N} \eta(k) + \lambda\eta(z(a,k))
$$
shows
$$
u_\delta(a) + \ts \frac{v^\lambda(k)}{N} - \cmd b_E(z(a,k)) - v^\lambda(z(a,k))
\ge \frac{1-\lambda}{N} \eta(k) + \lambda\eta(z(a,k)) \ge 0
$$
as desired, since $\bar v_\delta \ge v_\delta^t$.
This establishes $v_\delta = \bar v_\delta$ on $K$.
At $\bar k$,
convexity implies upper semicontinuity of $\bar v_\delta$ and it is 
dominated by the continuous function $v_\delta$,  so identity $v_\delta = \bar v_\delta$
extends to $\bar K$. 


As a consequence of \eqref{V' bound}--\eqref{V'' bound},
for $c_\delta>0$ both $v_\delta'$ and $v_\delta''$ 
are bounded away from zero
so the constraints
$\min\{v',v''\} \ge 0$ are not binding.  We claim $(u_\delta,v_\delta)$ must
also minimize the linear program \eqref{delta dual} even among
feasible pairs which do not satisfy these additional constraints.
To see this, we'll suppose the objective was lower at some other
feasible pair $(u,v) \in F_0$ and derive a contradiction.  If $u,v \in C^2(\bar A)$,
then the pair $(1-s)(u_\delta,v_\delta) + s(u,v) \in F_0$ also lowers the objective
for $s>0$ sufficiently small,  and inherits the strict convexity and monotonicity of
$(u_\delta,v_\delta)$ to produce the desired contradiction.  If $u,v \not\in C^2(\bar A)$,
the same contradiction will be obtained after approximating $(u,v)$ by a smooth feasible pair.  
We can at least assume $u$ and 
$v$ are continuous and bounded according to the proof of Theorem \ref{T:no gap}.
The Stone-Weierstrauss theorem then shows $u$ and $v$ can be approximated
uniformly by smooth functions $(\tus,\tvs)$ such that
$u + \sigma \le \tus \le u + 2\sigma$ and 
$v \le \tvs \le v + \sigma$ as $\sigma \to 0^+$.
In this case, $(\tus,\tvs) \in F_0$ follows from $(u,v) \in F_0$.
Convergence of the objective function to its limiting value as $\sigma \to 0$
is readily verified.
This establishes the desired contradiction,  hence the minimality of 
$(u_\delta,v_\delta)$ in $F_0$.


Now Corollary \ref{C:characterize} asserts there are
non-negative measures  $\epsilon_\delta \ge 0$ and $\lambda_\delta \ge 0$
satisfying the perturbed feasibility constraints \eqref{delta primal}
such that
\begin{equation}\label{delta optimal}
\alpha(u_\delta) + \delta \langle u_\delta + v_\delta \rangle_A
= \cmd \epsilon_\delta(\bet) + \lambda_\delta(\blt).
\end{equation}
%
Lemma \ref{L:compactness} yields a subsequential limit
$(u_{\delta_i},v_{\delta_i}) \to (u_0,v_0)$ pointwise on $\bar A \times \bar K$
and uniformly on compact subsets of
$[0,a_0[ \times [0, k_0[$,  with $u_0(a) = +\infty$ for $a>a_0 \in[0,\bar a]$
and $v_0(k)=+\infty$ for $k>k_0 \in [0,\bar k]$ and $a_0\le k_0$.
We claim $a_0=\bar a$.
Recalling the monotonicity of $u_\delta$, 
if $a_0<\bar a$ we have $u_{\delta_i}(k) \to  +\infty$ uniformly on
$a \in [(a_0+\bar a)/2,\bar a]$.  Since $\bar a \in \spt \alpha$,
Fatou's lemma will contradict
the bound \eqref{a priori bound} unless $a_0=\bar a$.
This also forces equality in $\bar k = \bar a \le k_0 \le \bar k$.
Thus $(u_0,v_0)$ are feasible for the original problem \eqref{three line dual}.


Extracting a further subsequence if necessary,  we may also assume
$(\epsilon_{\delta_i},\lambda_{\delta_i}) \to (\epsilon_0, \lambda_0)$ weak-$*$
in $C(\bar A \times \bar K)^*$ as $\delta_i \to 0$
to feasible measures for the dual problem \eqref{three line primal}.
(This compactness argument and topology are also described in the proof
 of Lemma \ref{L:primal existence}.)
Taking the $\delta \to 0$ limit of \eqref{delta optimal}, Fatou's lemma combines
with the weak-$*$ convergence to give
$$
\alpha(u_0) \le c_0\epsilon_0(\bet)+ \lambda_0(\blt) \in \R.
$$
Proposition \ref{P:easy duality} yields the opposite inequality,
and its corollary then confirms
the desired optimality of $(u_0,v_0)$ (and of $(\epsilon_0,\lambda_0)$).


Noting  $v_\delta = \bar v_\delta$, the inequalities
\eqref{V' bound}--\eqref{V'' bound}
survive passage to the $\delta_i \to 0$ limit in both the distributional \eqref{distributional} and a.e.\ senses.
For $j=1$ or $j=2$, when $N\theta^j \ge 1$,  these inequalities imply
$d^j v_\delta/dk^j \ge \ub b_L^{(j)}\min\{(1-\theta')^j,(\theta')^jN'\}$ throughout $K$ before and
hence after the limit.
It remains to show the identity $v_\delta = \bar v_\delta$
survives the $\delta_i \to 0$ limit first on $K$, and eventually on $\bar K$.


Although we have only subsequential
convergence of $(u_\delta,v_\delta)$,
we abuse notation by writing $\delta \to 0$ to denote this subsequence hereafter.
Taking $\delta \to 0$ in the remaining identity  of interest $v_\delta = \bar v_\delta$ yields
\begin{equation}\label{limsup identity}
v_0
:= \lim_{\delta \to 0} v_\delta
= \max\{\limsup_{\delta\to 0} v_\delta^w, \limsup_{\delta\to 0} v_\delta^m, \limsup_{\delta\to 0} v_\delta^t\}.
\end{equation}
Using $\bar k^-$ to denote the limit $k \uparrow \bar k$,
we claim $u_0(\bar k^-)<\infty$ if $v_0(\bar k^-)<\infty$,  and
$u_0(\bar k^-)=\infty$
if $v_0(\bar k^-)=\infty$.
The second claim follows from \eqref{stability bound},
which gives  $u_0(a) \ge \frac{N-1}N v_0(a)$;
the first claim is more subtle unless $\alpha$ has a Dirac mass at $\bar a$, but 
follows from the boundedness of $v_0$
in the supremum \eqref{student wage} due to the following parenthetical paragraph.

(To see that \eqref{student wage} continues to hold when $\delta=0$ assuming $\alpha[\{\bar a\}]=0$, 
 consider the continuous function 
$f_\delta(a,k) := u_\delta(a) + \frac1N v_\delta(k) - \cmd b_E(z(a,k)) - v_\delta(z(a,k)) \ge 0$
 on $A \times \bar K$.
 The zero set $Z_\delta$ of $f_\delta$ is relatively closed in $A \cap \bar K$;  it is non-decreasing by
 the strict submodularity shown in Lemma \ref{L:integral form},
 and contains $(A \times \bar K ) \cap \spt \epsilon_\delta$ according to Corollary \ref{C:characterize}.
For each $(a_\delta,k_\delta) \in Z_\delta$ this monotonicity implies
\begin{equation}\label{left right}
\int_{]a_\delta,\bar a] \times \bar K} \epsilon_\delta(da,dk)
\le \int_{\bar A \times [k_\delta,\bar k]} \epsilon_\delta(da,dk). 
\end{equation}
Fixing $a_\delta =a$,  to establish the limiting case of \eqref{student wage}
it is enough to show $\lim\sup_{\delta \to 0} k_\delta < \bar k$. Recalling
that the left and right marginals of $\epsilon_\delta$ are given by \eqref{delta primal},
setting $\Delta a=\bar a-a$ and $\Delta k_\delta = \bar k - k_\delta$, from \eqref{left right} 
we deduce
\begin{eqnarray*}
\frac1N (\bar a\alpha(]\bar a-\Delta a,\bar a]) + 
\delta \Delta a) &\le&
\bar a (z_\#\epsilon_\delta)([\bar k -\Delta k_\delta,\bar k]) + 
\delta \Delta k_\delta 
\\ &\le&  
\delta\Delta k_\delta + (\bar a \alpha + \delta H^1|_A)([\bar a - \frac1{1-\theta}\Delta k_\delta,\bar a])
\end{eqnarray*}
where the second inequality follows from \eqref{adult tail bound}.
Since $\bar a \in \spt \alpha$ but $\alpha[\{\bar a\}] =0$, 
the left hand side remains bounded away from zero in the limit $\delta \to 0$, whence
we conclude $\lim \inf_{\delta \to 0} \Delta k_\delta >0$ also. Thus \eqref{student wage} holds for $a \in A$
with $\delta=0$.)


Now if $v_0(\bar k^-)<\infty$
then Corollary \ref{C:uniform from below} allows us to deduce
$\ds \limsup_{\delta \to 0} v_\delta^t \le v_0^t$ for $k \in [0,\bar k[$ from
\begin{equation}\label{teachers wage}
v_\delta^t(k) = N \sup_{a \in [0,\bar a[} \cmd b_E(z(a,k)) + v_\delta(z(a,k)) - u_\delta(a),
\end{equation}
noting $u_0(a) \le \lim\inf_{\delta \to 0} u_\delta(a)$ uniformly on $[0,\bar a[$
and $z(a,k)$ is constrained to the range where the convergence
$v_\delta\to v_0$ is uniform.
Showing $\ds \limsup_{\delta \to 0} v_\delta^w \le v_0^w$
and       $\ds \limsup_{\delta \to 0} v_\delta^m \le v_0^m$ is similar but simpler,
whence $v_0 \le \max\{v_0^w,v_0^m,v_0^t\}$.
The opposite inequality follows from the constraints satisfied by $(u_0,v_0)$.

If $v_0(\bar k^-)=+\infty$ on the other hand,  then for fixed $k \in [0,\bar k[$
let $C_\delta$ denote the supremum of $\cmd b_E(z(a,k))+ v_\delta (z(a,k))$ over
$a \in [0,\bar a[$ and observe $C_\delta \to C_0<\infty$ as $\delta \to 0$.
Take $\delta_0>0$ sufficiently small that $C_{\delta_0} < 2C_0$,  and
smaller if necessary using Corollary \ref{C:uniform from below}
so that $u_\delta(\bar a-\delta_0)>2 C_0$ for all $\delta<\delta_0$.
For $\delta<\delta_0$,  the supremum \eqref{teachers wage} is unchanged
if we restrict its domain $a \in [0,\bar a-\delta_0]$ to an interval where
convergence $(u_\delta,v_\delta) \to (u_0,v_0)$ is uniform.  Thus taking $\delta\to 0$
in \eqref{teachers wage} yields $\ds \lim_{\delta \to 0} v_\delta^t(k) = v_0^t(k)$.
A similar but simpler argument yields $v_0^w(k) = \ds \lim_{\delta \to 0} v_\delta^w(k)$
and $v_0^m(k) = \ds \lim_{\delta \to 0} v_\delta^m(k)$,  whence the desired identity follows
from \eqref{limsup identity}.

It costs no generality to replace
$u_0$ by the right hand side of \eqref{perturbed student wage}
with $\delta=0$, which is feasible and no larger than $u_0$ in any case.
(In fact, they coincide throughout $A$ by the parenthetical paragraph above.)
Let us now argue that we may take $v_0$ to be continuous, or equivalently take equality
to hold in $v_0(\bar k^-) \le v_0(\bar k)$.  If $v_0(\bar k^-)<v_0(\bar k)$,  replacing
$v_0(\bar k)$ with $v_0(\bar k^-)$ does not violate any of the feasibility constraints.
Nor does it affect the values of $v_w,v_m,v_t$ or $u_0$ --- except to remedy any
discontinuity in $v_t$ or $u_0$ by reducing $v_t(\bar k)$ and $u_0(\bar a)$.
This can only improve the objective,  and by continuity of all of the resulting functions
extends the identity $v_0=\bar v_0$ from $K$ --- where it was already
established --- to $\bar K$, to complete the proof.
%
%
\end{proof}

\subsection{Uniqueness and properties of optimal matchings}

Finally,  we are ready to tackle the structure of optimal matchings
in the education and labor sectors,  and to give conditions guaranteeing
uniqueness of optimizers for both the primal and dual problems
\eqref{three line dual}--\eqref{three line primal}.

The structure our education sector often leads to positive assortative matching
$\epsilon$ of students with teachers. (Our labor sector always leads to
positive assortative matching of workers to managers.)  However, since
distribution $\kappa$ of cognitive skills acquired by adults in our population
is endogenous,  it might not be unique.  The following theorem specifies conditions
for uniqueness.  These require, in particular,  that $\kappa$ as well as the
exogenous distribution of student skills $\alpha$ be atom free.  The
following lemma details how $\kappa$ inherits this and other useful properties from
the distribution $\alpha$ of student skills input.  Even without positive assortativity,
unless the (exogenous) probability measure $\alpha$ concentrates positive mass at the
top skill type $\alpha[\{\bar a\}]>0$,  it follows that $\kappa$ concentrates no
mass at the upper endpoint of $K=[0,\bar k[$.  Then any matching
$\epsilon \ge 0$ on $\bar A \times \bar K$ which satisfies the steady-state constraint
$\frac{1}{N}\epsilon^{1} \le z_\#\epsilon$ must
concentrate all of its mass on $A \times K$.

\begin{lemma}[Endogenous distribution of adult skills]\label{L:adult tail bound}
%
%
Fix $\theta \in]0,1[$ and a Borel measure $\alpha \ge 0$ on $\bar A$ 
with $\alpha[\bar A]<\infty$ for
$A=[0,\bar a[$ with $\bar a>0$.
Set $K = [0, \bar k[ =A$ and $z(a,k)=(1-\theta)a+\theta k$.
If $\epsilon \ge 0$ on $\bar A \times \bar K$
has $\alpha = \epsilon^{1}$
as its left marginal, then for each $\bar k -\Delta k \in K$
the corresponding distribution
$\kappa =z_\# \epsilon$ of adult skills
satisfies
\begin{equation}\label{adult tail bound}
\int_{[\bar k - \Delta k,\bar k]} \kappa(dk)
\le 
\int_{[\bar a - \frac1{1-\theta}\Delta k,\bar a]}\alpha(da).
\end{equation}
Thus $\kappa$ has no atom at $\bar k$ unless $\alpha$ has an atom at $\bar a$.

In addition, if $\epsilon$ is positive assortative and
$\alpha$ has no atoms,  then  $\kappa$ has no atoms and
$\epsilon = (id \times k_t)_\#\alpha$ for some
non-decreasing map $k_t:\bar A \longrightarrow \bar K$.
uniquely determined $\alpha$-a.e.\ by $\kappa$.
Moreover,
if $\alpha(da) = \alpha^{ac}(a)da$ is given by a density $\alpha^{ac} \in L^1(A)$,
then $\kappa(dk)=\kappa^{ac}(k)dk$ is given by a related density $\kappa^{ac} \in L^1(K)$
satisfying
\begin{equation}\label{Monge-Ampere}
\alpha^{ac}(a)= 
\left(1 + \theta(k_t'(a)-1)\right)
\kappa^{ac}(z(a,k_t(a)))
\end{equation}
for Lebesgue-a.e. $a \in A$.  In this case
$\|\kappa^{ac}\|_{L^\infty(K)} \le \frac1{1-\theta} \|\alpha^{ac}\|_{L^\infty(A)}$.
\end{lemma}

\begin{proof}
The definition $\kappa = z_\#\epsilon$ yields
$\kappa([\bar k - \Delta k,\bar k]) = \epsilon[z^{-1}([\bar k -\Delta k,\bar k])]$.
Now
$
\bar k -\Delta k \le z(a,k) \le (1-\theta)a+\theta\bar k
$
implies $a \ge \bar k-\frac1{1-\theta} \Delta k$.
Thus
$$\ts \kappa([\bar k - \Delta k,\bar k])
 \le \epsilon\left([\bar a-\frac1{1-\theta}\Delta k,\bar a] \times \bar K\right)
= \alpha([\bar a-\frac1{1-\theta}\Delta k,\bar a])$$
which is the desired bound
\eqref{adult tail bound}.

For the measure $\epsilon$ to be positive assortative means its support $\spt \epsilon$
is non-decreasing.  Except possibly for a countable number of jump discontinuities,  this
support is then contained in the graph of some non-decreasing map
$k_t:\bar A \to \bar K$.  If $\alpha$ is free of atoms,  the countable set of $a$
where the jumps occur is a set of measure zero.  Then the formula
$\epsilon = (id \times k_t)_\# \alpha$
and uniqueness of $k_t$ are
well-known facts, established e.g.\ in Lemma 3.1 of \cite{AhmadKimMcCann11}
and the main theorem of \cite{McCann95}.  It follows that
$f(a)=z(a,k_t(a))$ is non-decreasing, and pushes $\alpha$ forward to $\kappa$.
By Lebesgue's theorem,
$f'(a) =1-\theta +\theta k_t'(a)$ exists
$H^1$-a.e.\ and enjoys the positive lower bound
$f'(a) \ge 1-\theta$.
Thus $f$ is one-to-one and there is an
inverse function $g:\bar K \longrightarrow \bar A$ with Lipschitz constant at most
$\frac1{1-\theta}$ such that $g(\bar f(k))=k$ for any non-decreasing
extension $\bar f:\bar K \longrightarrow \bar A$ of $f$ (to points where
$k_t(a)$ may not be differentiable).  For $K'\subset \bar K$ we have
$\kappa[K'] = \alpha[f^{-1}(K')] = \alpha[g(K')]$.  Taking $K'$ to consist
of any single point shows $\kappa$ has no atoms if $\alpha$ has no atoms.
Taking $K'$ to be an arbitrary set of Lebesgue measure zero shows
$\kappa$ absolutely continuous with respect to Lebesgue if $\alpha$
is absolutely continuous with respect to Lebesgue,  noting
$H^1[g(K')] \le \frac1{1-\theta} H^1[K']$.
The formula $\alpha^{ac}(a) = f'(a) \kappa^{ac}(f(a))$ then follows essentially
from the fundamental theorem of calculus, and is argued rigorously in \cite{McCann97}.
The bound
$\|\kappa^{ac}\|_{L^\infty(K)} \le \frac1{1-\theta} \|\alpha^{ac}\|_{L^\infty(A)}$
is a consequence, so the proof is complete.
\end{proof}

\begin{theorem}[Positive assortative and unique optimizers]\label{T:unique}
Fix $\cm\ge0$ and positive $\theta,\theta', N,N'$ and $\bar a$ 
with $\max\{\theta,\theta'\}<1\le N$.  
Set $A=[0,\bar a[$ and 
let  $\alpha$ be a Borel probability measure on $\bar A$ satisfying the doubling
condition \eqref{doubling} at $\bar a \in \spt \alpha$. Define
$z(a,k) =(1-\theta)a + \theta k$, $\bet=b_E\circ z$ and $\blt(a,k)=b_L((1-\theta')a+\theta'k)$
where $b_{E/L} \in C^1(\bar K)$ satisfy \eqref{utility bound 0}--\eqref{utility bound 2}.
If $\epsilon,\lambda \ge 0$ on $\bar A^2$
maximize the dual problem \eqref{three line primal},
then the labor matching $\lambda$ is positive assortative.  Moreover,
there exist a pair of maximizers $(\epsilon,\lambda)$
for which the educational matching $\epsilon$ is also positive assortative.
If there exist minimizing payoffs $(u,v) \in F_0$ 
for  the dual problem
\eqref{three line dual} 
which %
are non-decreasing and strictly convex,
(as for example
if either $\cm>0$ or $N\theta^2 \ge 1$),
then any maximizing $\epsilon$ and $\lambda$ are positive assortative.
If, in addition, $\alpha$ is free from atoms then the maximizing $\epsilon$ and $\lambda$ are unique.
If, in addition, hypotheses
(d)-(f) from Proposition \ref{P:specialization} hold,
then $u'$ and $v'$ exist and are uniquely determined
$\alpha$-a.e.\ and $(z_\#\epsilon)$-a.e.\ respectively.  If, in addition,
$\alpha$ dominates some absolutely continuous measure whose support fills $\bar A$,
and $(u_0,v_0) \in F_0$ is any other 
minimizer with 
$v_0: A \longrightarrow \R$ locally Lipschitz
then $u_0=u$ holds $\alpha$-a.e., meaning $u_0$ is unique.
\end{theorem}

\begin{proof}
Set $K=[0,\bar k[ = A$.  Existence of a maximizing pair
$(\epsilon,\lambda)$ is asserted by Lemma \ref{L:primal existence}.
Let us begin by showing that they are positive assortative under
the extra condition 
that minimizing payoffs $(u,v)$
exist for \eqref{three line dual} 
which are strictly convex.
Lemma~\ref{L:integral form}
asserts $v(z(a,k))$ is then strictly supermodular.

Set
$f(a,k) = u(a) + \frac{v(k)}{N} - \cm b_E(z(a,k)) - v(z(a,k)) \ge 0$
and $g(k',k) = v(k') + \frac{v(k)}{N} - \blt(k',k)) \ge 0$
on $\bar A \times \bar K$,  with the convention $f(\bar a,\bar k) \le 0$ if $v(\bar k) =+\infty$,
and vanishing if and only if $u(\bar a) = +\infty$ in addition.
Corollary~\ref{C:characterize} asserts
$\epsilon(f)=0$ and $\lambda(g)=0$ for any dual maximizers $(\epsilon, \lambda)$.
Thus $\epsilon$ and $\lambda$ must vanish outside the respective zero sets
$F \subset \bar A \times \bar K$ of $f$ and $G\subset \bar K^2$ of $g$.


When 
$f$ and $g$ are strictly submodular, 
then $F$ and $G$ are non-decreasing in the plane,
meaning $\lambda$ and $\epsilon$ are positive assortative. 
This strict submodularity follows from that of $-b_E(z(a,k))$ and
$-v(z(a,k))$.

Finally, assume in addition that $\alpha$ is atom free.
If $(\epsilon_i,\lambda_i)$ are dual maximizers,  for $i=0,1$,
then so is their average
$(\epsilon_2,\lambda_2):= ({\epsilon_0+\epsilon_1},{\lambda_0+\lambda_1})/2$.
Thus $\epsilon_2$ vanishes outside the non-decreasing set $F$,
as do $\epsilon_{0/1}$.  Similarly $\lambda_i$ all vanish outside
the same non-decreasing set $G$ for $i=0,1,2$.  This strongly suggests
the asserted uniqueness,  an intuition we now make precise.
Except perhaps for a countable number of vertical segments,
the non-decreasing set $F$ is contained in the graph of a non-decreasing
map $k_t:\bar A \longrightarrow \bar K$.  Any joint measure $\epsilon$
with $\epsilon^1 = \alpha$ cannot charge these vertical segments,
since this would imply $\alpha$ has atoms.  Since our maximizers $\epsilon_i$
vanishes outside the graph of $k_t$,  we conclude they must coincide with
the measure $(id \times k_t)_\#\alpha$ by Lemma 3.1 of \cite{AhmadKimMcCann11}.
This identification shows $\epsilon_0=\epsilon_1$.
The associated distributions $\kappa=z_\#\epsilon_0$ and $\kappa_t = (\epsilon_0)^2/N$
of adult and teacher skills are therefore also unique.
Moreover,  $\kappa$ is free from atoms, according to
Lemma \ref{L:adult tail bound}.

Let $\lambda_i^1$ and $\lambda_i^2$ be the left and right marginals
of each maximizer $\lambda_i\ge 0$ for the labor sector,  whose feasibility
implies $\lambda_i^1 + \lambda_i^2/N' = \kappa - \kappa_t$ is also atom-free.
Let $\Delta\lambda = \lambda_0 - \lambda_1$
denote the difference of the two maximizers.  Recall that both $\lambda_i$ ---
and hence $\Delta \lambda$ --- must vanish outside the same non-decreasing set $G$.
Just as before,
the non-decreasing set $G$ has at most countably many horizontal and vertical segments,
which $\lambda_i$ cannot charge since its marginals 
are free from atoms. Now the positive marginals
$\Delta \lambda^1_+ := ((\Delta \lambda)_+)^1 = ((\Delta \lambda)^1)_+$
and $\Delta \lambda^2_+$ of the difference must have the same mass,
since the atom-free condition precludes cancellations.  On the other hand,
feasibility implies
$N \Delta \lambda^1_+ - N \Delta \lambda^1_- + \Delta \lambda^2_+ - \Delta \lambda^2_-=0$,
which forces $N \Delta \lambda^1_+ = \Delta \lambda^2_-$ (and
$N \Delta \lambda^1_- = \Delta \lambda^2_+$).  Since these two measures have the
same mass,  $N\ne1$ produces a contradiction unless $\Delta \lambda=0$.
If $N=1$ so that all adults are teachers, then $\lambda_i=0$.
This establishes the uniqueness asserted for the dual problem.

Having established the existence of positive assortative maximizers
when $v$ is strictly convex,  we now turn to the case that strict convexity
fails.  According to Theorem \ref{T:minimizing wages},
this happens only when $c=0$ and $N\theta^2<1$,
so we can approximate this situation as a $c\to 0$ limit.
%
Let $(\epsilon_c,\lambda_c)$ and $(u_c,v_c)$ be the (non-negative)
optimizers described above for the problem with $c>0$, so that
$c\epsilon_c(\bet) + \lambda_c(\blt) = \alpha(u_c)$ according to
Remark \ref{R:characterize}.  Using
the Banach-Alaoglu theorem as in the proof of Lemma \ref{L:primal existence},
and the compactness results of Lemma \ref{L:compactness},
we extracting a subsequential limit
$(\epsilon_c,\lambda_c) \to (\epsilon,\lambda)$
in the weak-$*$ topology on $C(\bar A \times \bar K)^*$ and
$(u_c,v_c) \to (u,v)$ locally uniformly on $[0,a_0[$,
with $u(a)=+\infty=v(a)$ for all $a > a_0$.
The limiting pairs are feasible for the primal and dual problems
respectively,
and positive assortativity survives the limiting process \cite{McCann95}.
%
Fatou's lemma allows us to
take the subsequential limit of $c\epsilon_c(\bet) + \lambda_c(\blt) = \alpha(u_c)$
to arrive at $\lambda(\blt) \ge \alpha(u)$.
The reverse inequality is asserted by Proposition \ref{P:easy duality},
and confirms optimality of $(\epsilon,\lambda)$ by Corollary~\ref{C:characterize}.


We now address uniqueness of the primal minimizers.
Since $u$ and $v$ are strictly convex,  both are continuous functions
with one-sided
derivatives throughout $K$,  and two-sided derivatives except perhaps at
countably many points.  Define $u(\bar a) = \lim_{a \to \bar a} u(a)$ and $v(\bar k)$ similarly.
Since the measures $\alpha$ and $z_\#\epsilon$
have no atoms,  the asserted derivatives of $u$ and $v$ exist.
Denote the distribution
of workers and managers by $\kappa_w := \pi^1_\#\lambda$
and $\kappa_m := \pi^2_\#\lambda/N'$.
The projections of $\spt \epsilon$ through $\pi^1(a,k)=a$ and $\pi^2(a,k)=k$
are compact sets of full measure for $\kappa_w$ and $\kappa_m$ respectively.
Take $\dom v' \subset ]0,\bar k[$ by convention.
For each $k' \in \pi^1(\spt \lambda) \cap \dom v'$,
there is a unique $k \in \bar K$ with
$(k',k) \in \spt \lambda \subset G$.
The first-order condition $g_{k'}(k',k)=0$
then gives $v'(k')=(1-\theta')b_L'( (1-\theta)k'+\theta k)$;
by strict convexity of $b_L$ 
there cannot be two such $k$ without
differentiability of $v$ failing at $k'$.
This shows $v'$ to be uniquely determined by $\lambda$ throughout
$\pi^1(\spt \lambda) \cap \dom v'$ --- a set of full $\kappa_w$ measure.
A similar argument with the roles of $k'$ and $k$ interchanged shows
$v'(k) = N'\theta' b_L'({(1-\theta')k' + \theta' k})$ is uniquely determined by $\lambda$
on the set $\pi^2(\spt \lambda) \cap \dom v'$
containing $\kappa_m$-a.e.\ manager type~$k$.

To address $v'(k)$ for the teacher types $k$,
assume hypotheses (d)-(f) of Proposition \ref{P:specialization}.
For $k_1 \in \spt \kappa_t \cap \dom v'$,  that proposition
provides a recursive formula \eqref{iterate 0} asserting $k_2 \in \dom v'$,
and relating $v'(k_1)$ to $v'(k_2)$,  where $(a_1,k_1) \in \spt \epsilon$
and $k_2=z(a_1,k_1)$ is the skill of those
adults who were trained by type $k_1$ teachers.
The strict monotonicity of $v'(k)$ we have assumed implies $a_1$ and $k_2$
are unique.
The proposition also asserts
that after a finite number $d$ of iterations,  this recursion
terminates with an adult of skill $k_d$ who is willing to become a worker
or a manager, and whose wage gradient $v'(k_d)$ is therefore determined by
the considerations above.  Thus $v'(k_1)$ is uniquely determined
by $\epsilon,\lambda$, and \eqref{diverging sum}.  This establishes the
$\kappa$-a.e.\ uniqueness of the wage gradient $v'$.

Finally,  we turn to the net lifetime surplus $u(a)$ of student type $a \in ]0,\bar a[$.
For $a \in \pi^1(\spt \epsilon) \cap \dom u'$, there exists $k \in \bar K$
(which we'll show to be unique)
such that $(a,k) \in \spt \epsilon \subset F$.  The first-order conditions
for one-sided derivatives $\pm f_a(a^\pm,k) \ge 0$ give
$$
v'(z(a,k)^-)  + c b_E'(z(a,k)^-)
\ge \frac{u'(a)}{1-\theta} 
\ge v'(z(a,k)^+) + \cm b_E'(z(a,k)^+). 
$$
However,  the convexity of $v$ on $]0,\bar k[$ assert $v'(z^-) \le v'(z^+)$ and similarly
for $b_E$,  so both
$v$ and $b_E$ must be differentiable at $z(a,k)$ and equalities hold throughout.
Thus
$$
 v'(z) +\cm b_E'(z)  =  \frac1{1-\theta}u'(a).
$$
Since the left hand side is strictly increasing in $z$,
we find $z(a,k)$ and hence $k$ is unique.
Since $v'$ was uniquely determined for $z_\#\epsilon$ adult type,
it follows that $u'$ is uniquely determined for $\alpha$-a.e. student type.
If 
$\alpha$ dominates some absolutely continuous measure whose support fills
$\bar A$, this shows $u$ is unique up to an additive constant.
Given another feasible minimizer $(u_0,v_0)$ with $v_0$ locally Lipschitz,
we see $u_0$ must produce equality $\alpha$-a.e.\ in the inequality
\eqref{perturbed student wage};  otherwise replacing $u_0$ by the 
right-hand side would remain feasible and lower the objective \eqref{three line dual}.
On the other hand,  the right hand side is locally Lipschitz,  according to
Lemma~\ref{L:inherit Lipschitz and semiconvex}.  The arguments above then yield
$u_0 = u + const$.   But the constant must vanish since both minimizers yield
the same value for the objective functional, showing $u_0$ is unique in $L^1(\bar A,\alpha)$.
\end{proof}


\subsection{
Phase transition to unbounded wage gradients}
\label{S:phase transition}

Having come this far,  one may wonder whether establishing the existence
of competitive equilibria need be so involved.  If we had
been content to find optimizing wages $u$ and $v$ which are merely non-decreasing,
an 
argument based on Helly's selection theorem
might have sufficed.  However, we
would not then know the convexity of the wages 
(used to prove their uniqueness),  nor
positive assortativity of the education sector.

In this section,  we explore the actual behavior of $v(k)$ near the top skill
type $\bar k$,  assuming the distribution
of student types is given by a continuous density
$\alpha(da) = \alpha^{ac}(a)da$ on $A=K=[0,\bar k[$.
Under mild differentiability hypotheses,
our next theorem establishes
the existence of a phase transition separating bounded from
unbounded wage gradients.
For $N\theta> 1$, it shows the education sector may form into a pyramid
scheme in which the marginal wage $v'(k)$ diverges to infinity as $k \to \bar k$,
even though the absolute wage $v(k)$ remains bounded.
For $N\theta \ne 1$,
it gives precise asymptotics \eqref{rate of divergence} for the wage function
$v(k)$ and the endogenous distribution $\kappa^{ac}(k)$ of adult skills near $\bar k$.
Notice this formula makes an explicit quantitative prediction for the dependence of the
rate of divergence on the teaching capacity $N$ and effectiveness $\theta$ assumed
in the model.
In all cases this divergence is integrable, so the wages tend to a finite limit.
For $N\theta<1$ it predicts a specific limiting slope $v'(k) \to \cm/(\frac{1}{N\theta}-1)$
as $k \to \bar k$, while for $N\theta>1$ it predicts $v'(k) \to \infty$  at a specific rate.
Thus the differences in marginal wages amongst the very top echelons of teachers (`gurus')
is negligible
in a thin (or equivalently, vertical) pyramid $N\theta<1$, but becomes more and more exaggerated
if $N\theta>1$, and at a rate which increases with $N\theta$,
corresponding to a fatter and fatter (or equivalently, more and more
horizontal) organizational structure with wider effective span of control.
When the theorem applies, it also predicts that the density of adults (= teachers)
at the highest skill level $\bar k=\bar a$
tends to a constant multiple $\frac{1-\theta/N}{1-\theta}$ of the density of students.


\begin{theorem}[Wage behavior and density of top-skilled adults]
\label{T:phase transition}
Fix $\cm\ge0$ and positive $\theta,\theta', N,N'$ and $\bar a=\bar k$ 
with $\max\{\theta,\theta'\}<1\le N$.  
Let  $\alpha$ be given by a Borel probability density $\alpha^{ac} \in L^\infty(A)$ which is continuous
and positive at the upper endpoint of $A=[0, \bar a[$.
Set $z(a,k) =(1-\theta)a + \theta k$, $\bet = b_E\circ z$ and $\blt(a,k)=b_L((1-\theta')a+\theta'k)$,
where $b_{E/L} \in C^1(\bar K)$ satisfy \eqref{utility bound 0}--\eqref{utility bound 2}.
Suppose $(\epsilon,\lambda)$
and convex $(u,v)\in F_0$ optimize the primal and dual problems
\eqref{three line dual}--\eqref{three line primal},
and
(i) $\bar k  \in (\spt \epsilon^2) \setminus \spt(\lambda^1 + \lambda^2)$,
meaning all adults with sufficiently high skills 
become teachers;
(ii) the educational matching $\epsilon$ is positive assortative,
meaning a non-decreasing correspondence $k=k_t(a)$ relates the ability
of $\alpha$-a.e.\ student $a$ to that of his teacher;
(iii) $k_t$ is differentiable at $\bar a$,
and (iv) $v$ is differentiable on some interval $]\bar k-\delta,\bar k[$.
Then for $N\theta \ne 1$,
\begin{equation}\label{rate of divergence}
v'(k) = \frac{const}{|\bar k - k|^{\frac{\log N\theta}{\log N}}} -
\frac{\cm \bar b_E'}{1-\frac{1}{N\theta}} + o(1)
\end{equation}
as $k \to \bar k$, and the steady state
distribution $\kappa=z_\#\epsilon$ of adult skills satisfies
\begin{equation}\label{guru density}
 \kappa^{ac}\left(\bar k\right) := \lim_{\delta \to 0}
\frac{1}{\delta} \int_{\bar k-\delta}^{\bar k} \kappa(dk) =
\frac{1-\theta/N}{1-\theta} \alpha^{ac}(\bar a).
\end{equation}
\end{theorem}

\begin{proof} 
As in Lemma \ref{L:adult tail bound},
hypothesis (ii) implies some non-decreasing function
$k_t:A \to K$ gives the equilibrium matching of students with teachers, so that
$k_g(a) = (1-\theta)a + \theta k_t(a)$ gives the matching of student ability with
human capital acquired when the student grows up.  Then $(k_g)_\#\alpha = \kappa$ and
$(k_t)_\# \alpha = N\kappa_t$,  where $\kappa= \kappa_m + \kappa_m/N' + \kappa_t/N$
gives the distribution of adult skill types on $K$, as a sum of the distributions
of worker, manager and teacher skill types.  Now
\begin{eqnarray}\label{MAt}
N k_t'(a) \kappa^{ac}_t(k_t(a)) &=& \alpha^{ac}(a), \\
{\rm and} \qquad k_g'(a) \kappa^{ac}(k_g(a)) &=& \alpha^{ac}(a)
\label{MAg}
\end{eqnarray}
is known to hold for a.e.\ $a\in\bar A$.  In particular, techniques of
\cite{McCann97} can be used to show it holds at $a=\bar a$ provided $k_t'(\bar a)$
(and hence $k_g'(\bar a)$) exists (iii) and are non-vanishing.
On the other hand, the upper bound $\|\kappa^{ac}\|_{L^\infty}<\infty$
from Lemma \ref{L:adult tail bound} gives a positive lower bound for $k_t'(a)$
near $\bar a$ a.e.\ in \eqref{MAt},
which precludes the possibility that $k_t'(\bar a)=0$.

From (i) and the steady state constraint
$\kappa = \lambda^1 + \frac1{N'}\lambda^2 + \frac1N \epsilon^2$ we have
$k_t(\bar a)=\bar k=k_g(\bar a)$ and
$\kappa^{ac}(\bar a) = \kappa^{ac}_t(\bar a)$.
From \eqref{MAt}--\eqref{MAg} we conclude $Nk_t'(\bar a) = k_g'(\bar a)$.
On the other hand,
differentiating $k_g(a) = (1-\theta)a+\theta k_t(a)$ yields
$k_g'(\bar a)=1+ \theta(k_t'(\bar a)-1)\ge 1-\theta$.
Solving this linear system of two equations in two unknowns
gives $k_t'(\bar a) = \frac{1-\theta}{N-\theta}$
and
\begin{equation}\label{kgp_max}
k_g'(\bar a) = \frac{1-\theta}{1-\theta/N}; 
\end{equation}
\eqref{MAg} now implies \eqref{guru density}.

Next we consider the equilibrium wage $v(k)$ of each type of adult
and payoff $u(a)$ to each type of student.
The stability constraint asserts
$u(a) + {\ts \frac{1}{N}} v(k) - v(z(a,k)) - \cm b_E(z(a,k)) \ge 0
$
for all $a$ and $k$,  with equality holding when $k=k_t(a)=\theta^{-1}k_g(a)+ (1-\frac1\theta)a$.
The first-order condition in $k$ for this non-negative function to attain its
minimum gives
$$
v'\left(\frac{k_g(a)-(1-\theta)a}{\theta} \right) = 
(v'(k_g(a)) + \cm b_E'(k_g(a))) N\theta.
$$
Taylor expanding $k_g(\bar a - \Delta a) = \bar k - k_g'(\bar a)\Delta a + o(\Delta a)$
using $k_g'(\bar a)$ from \eqref{kgp_max}, 
we find a recursive relation for $v'(k)$ near $\bar k$:
$$\ts v' \left( \bar k - \frac{1-\theta}{N-\theta} \Delta a+ o(\Delta a)\right)
= N\theta [v'+ \cm b_E']_{k=\bar k - \frac{1-\theta}{1-\theta/N} \Delta a+ o(\Delta a)}.
$$
Neglecting the $o(\Delta a)$ terms and setting
$\bar b_E' f(x):= v'(\bar k - x) / c+(1-\frac{1}{N\theta})^{-1}\bar b_E' - 
(1 - \frac{1}{N^2\theta})^{-1} b_E''(\bar k) x$,
the recursion simplifies to $f(\frac{x}{N}) = N\theta f(x)$ which is solved by
constant multiples of $f(x) = x^{- \log (N\theta)/ \log N}$.  Thus, to leading order
$$
v'(\bar k - \Delta k) 
= const |\Delta k|^{-\frac{\log N\theta}{\log N}} - \frac{c \bar b_E'}{1-\frac{1}{N\theta}} + \frac{cb_E''(\bar k)}{1-\frac{1}{N^2\theta}}\Delta k.
$$
Either the first or the second summand dominates this expression as $\Delta k \to 0$,
depending on the sign of $N\theta-1$.  One might worry that $const$ depends on the sequence 
along which the recursion is solved,  but for $N\theta>1$ the monotonicity
of $v'$ precludes this, 
to yield the desired identity \eqref{rate of divergence}.
\end{proof}

Some remarks 
concerning hypotheses (i)--(iv): 
Proposition \ref{P:specialization} ensures (i) holds if
$N'\theta'$ and $N\theta$ are large enough,  while Theorem \ref{T:unique}
ensures (ii) holds when $c>0$, and can be selected otherwise.
We do not know conditions which guarantee (iii)-(iv),
since differentiability may fail for $k_t(a)$ on a set of zero measure,
and for $v(k)$ at a countable number of points.  We can however,
ensure that $k_t$ is bi-Lipschitz by combining the lower bound on
its derivative from Lemma~\ref{L:adult tail bound} with the upper bound
provided by Proposition \ref{P:specialization} in case $N\theta \ge 1$.
This makes failure of (iii) seem unlikely,  since the value of
$k'_t(a)$ would have to oscillate between these positive bounds,
producing a reciprocal oscillation in $\kappa(k)$ near $\bar k$.
Similarly, the alternative to (iv) is that jump discontinuities
in the monotone function $v'(k)$ accumulate at $\bar k$.
At least one of the three types of singular behavior must occur,
and \eqref{rate of divergence} seems the most likely,
especially given its consistency with the divergence~\eqref{diverging sum}
predicted by Proposition \ref{P:specialization}.
To be absolutely correct, however,
one should say Theorem \ref{T:phase transition} provides strong evidence
in favor of a phase transition with wage gradients diverging if and only if $N\theta \ge 1$,
where the leading order behavior of \eqref{rate of divergence} changes.
The theorem also provides concrete quantitative predictions
which can be investigated numerically.


\appendix
\section{Optimal plans and absence of a duality gap}
\label{S:no duality gap}

This appendix establishes the existence of measures achieving
the maximum $LP^*(\delta)$ in the original \eqref{three line primal} and
$\delta$-perturbed dual problem \eqref{delta primal},  and verifies
the absence $LP^*(\delta) = LP_*(\delta)$ of a duality gap.
While such claims are natural analogs to duality results 
well-known in finite-dimensional linear programming,
in our infinite-dimensional context 
they will remain true only if we are careful to choose the correct functional
analytic setting.  These choices are made clear in the proofs of the following 
statements.


\begin{lemma}[Existence of optimal measures]\label{L:primal existence}
Fix $\delta,\cmd$ non-negative and $\theta,\theta',N,N' $ 
positive with 
$\max\{\theta,\theta'\}\le 1 \le N$
and $N \ge 1$.  Let  $\alpha$ be a Borel probability measure on $\bar A$,
where $A=[0, \bar a[=K$ with $0<\bar a=\bar k \in \spt \alpha$, and define
$z(a,k) =(1-\theta)a + \theta k$, $\bet=b_E \circ z$ and $\blt(a,k)=b_L((1-\theta')a+\theta'k)$,
where $b_{E/L} \in C^0(\bar K)$. 
%
Then there exist feasible measures $\epsilon_\delta \ge 0$ and $\lambda_\delta \ge 0$
on $\bar A^2$ maximizing the dual problem \eqref{delta primal}.
\end{lemma}

\begin{proof}
As we now describe,
existence of a maximizing $\epsilon$ and $\lambda$ follows from a standard compactness
and continuity argument.  The continuous functions $C(\bar A^2)$ on the compact
square $\bar A^2$ form a Banach space when equipped with the supremum norm $\|\cdot \|_\infty$.
Borel probability measures form a weak-$*$ compact subset of the dual Banach space,
according to the Riesz-Markov and Banach-Alaoglu theorems.  A sequence
$\epsilon_i \to \epsilon_\infty$ converges in the weak-$*$ topology if and only if the integral
$\epsilon_i(f)$ of each continuous function $f \in C(\bar A^2)$ against $\epsilon_i$ converges to
the integral of $f$ against $\epsilon_\infty$.  Feasibility of $\lambda,\epsilon \ge 0$
asserts
\begin{eqnarray*} 
\langle f\delta \rangle_A + \int_{\bar A} f(a) \alpha (da)
&=&\int_{\bar A \times \bar K} f(a) \epsilon(da,dk) \quad {\rm and}
\\
\int_{\bar K^2} [f(k') + \ts\frac{1}{N'}\ds f(k)] \lambda(dk',dk)
&=& 
\langle f\delta \rangle _K + \int_{\bar A \times \bar K} [f(z(a,k)) - \ts\frac{1}{N}\ds f(k)] \epsilon(da,dk)
\end{eqnarray*}
for each $f \in C(\bar A)$.  Thus the feasible pairs form a weak-$*$ compact
subset of $C(\bar A^2)^*$.  Since $b_\theta,b'_{\theta'} \in C(\bar A^2)$,  the linear
functional we are trying to maximize is weak-$*$ continuous,  hence its maximum
must be attained, provided
the set of feasible measures $(\epsilon,\lambda)$ is non-empty.
To see the feasible set is non-empty, let $\epsilon$ concentrate on
the diagonal: $\epsilon = (id \times id)_\# (\alpha + \frac{\delta}{|A|} H^1|_A)$.  Then the
marginals $\epsilon^1=\epsilon^2$ of $\epsilon$ coincide with
$\kappa := z_\#\epsilon = \alpha + \frac{\delta}{|A|} H^1|_A$,  since $z(a,a) =a$.  Choosing
$\lambda := \frac{1-1/N}{1+1/N'} \epsilon + \frac{1}{1+1/N'}(id \times id)_\# (\frac{\delta}{|A|} H^1|_A)$
defines a feasible pair.
\end{proof}

\medskip

The next theorem addresses the absence of a duality gap.  It is proved using
generalization of the Fenchel-Rockafellar duality theorem
found in Borwein and Zhu \cite{BorweinZhu05}
(and pointed out to us by Yann Brenier).
As in the preceding lemma,  the Fenchel-Rockafellar theorem
will involve the duality between measures and continuous, bounded functions.
On the other hand, $LP_*(\delta)$ is necessarily defined by an infimum
over a larger class of functions $F_\delta$ including some unbounded ones. 
Thus the Fenchel-Rockafellar theorem by itself yields only an inequality
$LP_*(\delta) \le LP^*(\delta)$ and not the desired equality.
Fortunately,  the complementary inequality
is established in Proposition \ref{P:easy duality}.


\begin{theorem}[No duality gap]
\label{T:no gap}
Fix $\delta,\cmd$ non-negative and $\theta,\theta',N,N'$ and $\bar a=\bar k$ positive with 
$\max\{\theta,\theta'\}\le 1 \le N$.
Let  $A=[0,\bar a[=K$ and
$\alpha$ be a Borel probability measure on $\bar A$
satisfying the doubling condition \eqref{doubling} at $\bar a$,
and define
$z(a,k) =(1-\theta)a + \theta k$, $\bet=b_E\circ z$ and $\blt(a,k)=b_L((1-\theta')a+\theta'k)$  where $b_{E/L} \in C(\bar K)$.
%
Then the values $LP^*(\delta)=LP_*(\delta)$ of the infimum
\eqref{delta dual} and supremum \eqref{delta primal} coincide.
\end{theorem}

\begin{proof}
Let $H:Z \longrightarrow Z^*$ be a bounded linear transformation
between a Banach space $Z$ and its dual $Z^*$,  on which
convex functions
$\varphi:Z \longrightarrow \R \cup \{+\infty\}$ and
$ \phi:Z^* \longrightarrow \R \cup \{+\infty\}$ are defined.
Let $\dom \varphi := \{ z\in Z \mid \varphi(z)<\infty\}$.
Define the Legendre transform $\phi^*$ of $\phi$ by
\begin{equation}\label{Fenchel-Rockafellar}
\phi^*(z) := \sup_{z^* \in Z^*} \langle z,z^* \rangle - \phi(z^*)
\end{equation}
on $z \in Z$ and analogously $\varphi^*$ on $Z^*$.  Here $\langle z,z^*\rangle$
denotes the duality pairing.
If $\phi$ is continuous and real-valued at some point in $H(\dom\varphi)$,
then pp. 135-137 of \cite{BorweinZhu05} asserts
$$
\inf_{z \in Z} \varphi(z) + \phi(Hz) = \max_{z^* \in Z^*} - \varphi^*(H^*z^*) - \phi^*(-z^*).
$$

In our case
$$\varphi_\delta(u,v) 
=
\delta \langle u + v\rangle_A + \int_{[0,\bar a]} u(a) \alpha(da)$$
so
$$
\varphi_\delta^*(\mu,\nu) 
= \left\{
  \begin{matrix} 0 & {\rm if}\ (\mu,\nu)
  =(\alpha + \frac{\delta}{|A|}H^1|_A ,\frac{\delta}{|K|}H^1|_K) \\
                +\infty & {\rm else,}
  \end{matrix}\right.
$$
while
$$\phi(\tilde u,\tilde v) =
\left\{
  \begin{matrix} 0 & {\rm if}\ \tilde u \ge \cmd \bet\ {\rm and}\ \tilde v \ge \blt \\
                +\infty & {\rm else;}
  \end{matrix}\right.
$$
so
$$\phi^*(\epsilon,\lambda) =
\left\{
  \begin{matrix}
     \ts \cmd \epsilon (\bet) + \ds \lambda(\blt) & {\rm if}\ \epsilon \le 0\
      {\rm and}\ \lambda \le 0
   \\ +\infty & {\rm else;}
  \end{matrix}\right.
$$
and $H:C(\bar A)\oplus C(\bar K) \longrightarrow C(\bar A \times \bar K) \oplus C(\bar K\times \bar K)$
is given by
$$H \left(\begin{matrix} u \\ v \end{matrix}\right)
= \left( \begin{matrix}
u(a)+\frac{1}{N}v(k) - v(z(a,k)) \\
v(k') + \frac{1}{N'} v(k)
\end{matrix}\right)
,$$
so that
$$H^* \left(\begin{matrix} \epsilon \\ \lambda \end{matrix}\right) =
\left(\begin{matrix} \epsilon^{1} \\
\lambda^{1} +\frac{1}{N'}\lambda^{2} + \frac1N \epsilon^{2} - z_\# \epsilon
\end{matrix}\right).
$$
Notice $\varphi$ is continuous,  while taking $u,v$ large and constant makes
$\phi \circ H$ finite.  With these definitions \eqref{Fenchel-Rockafellar}
therefore asserts:
\begin{eqnarray*}
LP_*(\delta)
&\le& \inf_{u \in C(\bar A) \atop v \in C(\bar K)} \varphi_\delta(u,v) + \phi(H(u,v))
\\ &=& \max_{\epsilon \ge 0\ {\rm on}\ \bar A \times \bar K\atop
         \lambda \ge 0\ {\rm on}\ \bar K \times \bar K}
- \varphi_\delta^*(H^*(\epsilon,\lambda)) - \phi^*(-\epsilon,-\lambda)
\\ &=& LP^*(\delta).
\end{eqnarray*}
Here we have an inequality rather than the desired equality because
the definition of $LP_*(\delta)$ involves minimizing over a
broader class of feasible functions 
\eqref{delta dual} which need neither be continuous nor bounded.
For such functions however, Proposition \ref{P:easy duality}
asserts the opposite inequality, to conclude the proof of the theorem.
\end{proof}


\end{document}